%% file: ms.tex
\documentclass[final,onefignum,onetabnum]{siamart171218}

\input{ex_shared}

\headers{Fast bilinear algorithms for convolution}{Ju and Solomonik}

\title{Derivation and analysis of fast bilinear algorithms for convolution\thanks{Submitted to the editors November 20th, 2019.}}

\author{Caleb Ju\thanks{Department of Computer Science, University of Illinois at Urbana-Champaign (\email{calebju2@illinois.edu}, \email{solomon2@illinois.edu}).} 
\and 
Edgar Solomonik\footnotemark[2]} 
\ifpdf \hypersetup{ pdftitle={Derivation and analysis of fast bilinear algorithms for convolution},
pdfauthor={} } \fi

\myexternaldocument{ex_supplement}
\usepackage{subcaption}
\DeclareMathAlphabet{\pazocal}{OMS}{zplm}{m}{n}

\begin{document}

\maketitle

\begin{abstract} 
The prevalence of convolution in applications within signal processing, deep
    neural networks, and numerical solvers has motivated the development of
    numerous fast convolution algorithms. In many of these problems,
    convolution is performed on terabytes or petabytes of data, so even
    constant factors of improvement can significantly reduce the computation
    time.  We leverage the formalism of bilinear algorithms to describe and
    analyze all of the most popular approaches.  This unified lens permits us
    to study the relationship between different variants of convolution as well
    as to derive error bounds and analyze the cost of the various algorithms.
    We provide new derivations, which predominantly leverage matrix and tensor
    algebra, to describe the Winograd family of convolution algorithms as well
    as reductions between 1D and multidimensional convolution. We provide cost
    and error bounds as well as experimental numerical studies. Our
    experiments for two of these algorithms, the overlap-add approach and
    Winograd convolution algorithm with polynomials of degree greater than one,
    show that fast convolution algorithms can rival the accuracy of the fast
    Fourier transform (FFT) without using complex arithmetic. These algorithms
    can be used for convolution problems with multidimensional inputs or for
    filters larger than size of four, extending the state-of-the-art in
    Winograd-based convolution algorithms.
\end{abstract}



\section{Introduction} \label{sec:intro}
Discrete convolution is a bilinear function that combines two sequences of data
to produce a third. Problems such as
multiplication~\cite{harvey2016even,schonhage1971schnelle,fnt}, signal
processing~\cite{ifir,filter1,filter2,texture},
statistics~\cite{nguyen2014,lee2005efficient},
acoustics~\cite{adriaensen2006design,quate1970convolution},
geophysics~\cite{rice1962inverse}, molecular
simulation~\cite{pierce2011accelerating}, image
processing~\cite{lavin,pratt1974correlation}, and numerical solvers for partial
differential equations within physics and
chemistry~\cite{pde1,pde2,khoromskij2010fast} use convolution.  Consequently,
fast methods for convolution can reduce the computation time for various
problems.  Given two inputs of size $n$, a direct computation of convolution
performs at most $n(n-1)$ additions and $n^2$ multiplications. 

Over the years, \emph{fast algorithms} have been studied and used to 
compute convolution. Most fast algorithms operate in three steps: compute 
the linear combinations of both inputs, calculate the element-wise product of
those linear combinations, and then recover the result by computing the linear
combinations of the products. The first known fast algorithm is Karatsuba's
algorithm~\cite{karatsuba1995complexity}, which achieves a complexity of
$O(n^{\log_2(3)})$.  The most prominent fast algorithm employs the discrete
Fourier transform (DFT) to obtain suitable linear combinations. This fast
algorithm leverages the fast Fourier transform (FFT) to obtain the linear
combinations in $O(n\log(n))$ time~\cite{harvey2016even, harvey2019polynomial}.
The FFT-approach reduces the number of bilinear products necessary to $O(n)$,
yielding an algorithm with an overall cost of $O(n\log(n))$ instead of the
$O(n^2)$ cost incurred by the direct method.

For a convolution with two $n$-dimensional vectors, the cost and stability of
the FFT make it the method of choice. However, in many scenarios, including
signal processing and convolutional neural networks (CNN), a small {\it filter}
of size $r$ is convolved with a large vector of size $n$. A naive
application of the FFT requires $O(n\log(n))$ cost, which is worse than the
$O(nr)$ cost of the direct method when $r<\log(n)$. The use of $n/r$ FFTs of
size $O(r)$ yields a lower cost of $O(n \log(r))$. Furthermore, when $r$ is
small, the constant factors incurred by FFT, due in part to the use of complex
arithmetic, can make it uncompetitive~\cite{fialka2006fft}. Given a direct
implementation of complex arithmetic, an FFT-based convolution with $n$-dimensional
vectors requires $18n\log(2n) + O(n)$ real additions and $12n\log(2n) + O(n)$
real multiplications. For sufficiently small dimensions, the direct approach
requires less work than the FFT. While the direct approach is efficient in such
cases, other fast algorithms can obtain yet lower constant factors,
yielding practical benefits. Consider the use of convolution in CNNs. The
convolutional layer of the CNN architecture AlexNet~\cite{alex} takes 
approximately three hours, about ninety percent of the CNN's overall run time, 
to convolve $256$ images when running on a single-threaded CPU~\cite{cong}.
Even a constant factor improvement over the direct method can save minutes to
hours for this type of problem. Fast algorithms present a variety of
methods with lower cost complexities.

Beyond adaptation for small filters, another remaining challenge is the
development of efficient methods for multidimensional (especially, 2D and 3D)
convolution algorithms. Efficient algorithms for 2D and 3D convolution are
important for applications within scientific computing and CNNs. The FFT-based
approach is well-suited for the former domain, but the use of small filters in
CNNs again leaves room for further innovation. To the best of our knowledge,
the main algorithms for computing convolution in CNNs are either
matrix-multiplication~\cite{cudnn}, the FFT~\cite{fftw}, or a few variants of
Winograd's convolution algorithm~\cite{winograd,lavin,beyond}. We propose other
variants of the general Winograd formulation that are suitable for higher
dimensions, such as 2D, 3D, and 4D convolutions.

\subsection{Previous surveys and key related work}
Convolution has been studied and surveyed before in signal 
processing~\cite{survey7, survey2, survey5, survey3, survey4, survey1}. Some of
these methods have been presented as bilinear algorithms, which provide a
framework to define new algorithms for larger convolutions via a matrix nesting
by the Kronecker product and embedding using the Winograd convolution
algorithm. In addition to using the formalism of bilinear algorithms to define
these methods, we provide explicit formulations on how to generate the matrices
for the various convolution algorithms. We provide new, simple derivations for
many of the key methods, and specially address multidimensional and
small-filter convolution scenarios.

An important consideration for bilinear algorithms is the number of additions
and element-wise multiplications required to compute the algorithm. The cost of
applying the linear combinations scales quadratically to the input size.
Variations of bilinear algorithms for convolution offer trade--offs between 
the number of linear combinations and element-wise multiplications
needed~\cite{survey3}, which has subsequently been studied and optimized for
various implementations of convolution algorithms~\cite{opt}. We provide
similar tables as well as supplementary
material\footnote{\url{https://github.com/jucaleb4/Bilinear-Algorithms-for-Convolution}}
for readers to generate the matrices themselves.

With the advent of parallel computing, the scalability of convolution
algorithms is crucial for building highly efficient algorithms. The
parallelization of convolution with Sobel or Gaussian filters has been
studied~\cite{survey8, gauss}. Sobel filters are used for
edge-detection~\cite{survey8, kanopoulos1988design} and Gaussian filters are
used for reducing the noise in a signal~\cite{gauss,deng1993adaptive}. However,
a more general study of the parallel efficiency of convolution algorithms may
be useful as filters in CNNs are not restricted to the Sobel or
Gaussian variants. In the context of CNNs, a direct computation of discrete
convolution is fairly straight-forward to parallelize~\cite{krizhevsky2014one}.
Convolution can be reduced to a matrix-multiplication and fast Fourier
transform problem, both of which can leverage efficient library packages, such
as cuDNN~\cite{cudnn} for direct convolution on GPUs and FFTW~\cite{fftw} for
FFT on shared-memory machines (but also many other for GPUs, shared-memory, and
distributed-memory architectures). The family of fast convolution algorithms
from signal processing (aside from the FFT) has been largely unused for CNNs
prior to the paper by Lavin and Gray~\cite{lavin}. They propose a method based
on Winograd's formulation of convolution algorithms, although it is later noted
to be a variant of the Toom-Cook (interpolation) method~\cite{err1}.  The
Winograd-based algorithm~\cite{lavin} divides the convolution into three steps,
each step performing a sequence of matrix multiplications. Experiments on GPUs
suggest that the Winograd-based algorithm can be highly scalable for
small-filter convolution problems~\cite{lavin}. In general, both the FFT and
Winograd-based method achieve comparable execution times. When executed on
GPUs, the FFT and Winograd method achieve speed-ups of up to $4\times$ over the
direct (matrix-multiplication-based) approach for AlexNet~\cite{perform1}.
Additionally, parallel implementations~\cite{par1,par2} and specialized
hardware designs~\cite{hardware1, ifir, hardware3} have been shown to improve
the speed of convolution. We focus on the sequential arithmetic complexity and
stability of fast algorithms for convolution.

The use of linear combinations in fast algorithms leverages cancellation
of sums to reduce the number of element-wise multiplications. However, it 
may also introduce significant error for certain inputs. Since the Winograd-based
convolution~\cite{lavin}, a modified Toom-Cook method, relies on the Vandermonde
matrix, the algorithm can quickly become inaccurate for inputs of size
greater than four~\cite{lavin,winograd}. The absolute error of the Toom-Cook
method is proportional to norm of the inputs and the three matrices that
compose the bilinear algorithm~\cite{err1}. Better nodes and the use of the
Winograd convolution algorithm with polynomials of degree greater than one have
shown promising results in reducing error~\cite{err1,beyond}. In addition to
summarizing these results, we show that decomposing a large convolution into
nested smaller convolutions can result in more stable algorithms.

Given the wide array of work for convolution in signal processing and more
recently CNNs, our main contribution is to provide a comprehensive guide and
derive simple constructions for the various convolution algorithms. To do so,
we leverage the general bilinear algorithm formulation, which enables
derivation and analysis of fast algorithms using basic algebraic
transformations of matrices and tensors.

\subsection{Convolution and its variants} 
The \emph{convolution} of two continuous functions $\fun{u}$ and $\fun{v}$ is 
defined as
\begin{equation} \label{eq:convolution} 
    (\fun{u} \ast \fun{v})(t) = \int_{-\infty}^{\infty} \fun{u}(\rho) \fun{v}(t-\rho) d\rho.  
\end{equation}
Given the input vectors $\vcr{f} \in \mathbb{R}^r$ and $\vcr{g} \in
\mathbb{R}^n$ (assume $n \ge r$), the discrete convolution between $\vcr f$
and $\vcr g$ is defined as
\begin{equation} \label{eq:discrete_convolution} 
    y_k = \sum_{i} f_i g_{k-i}.
\end{equation}
We leave the starting and ending indices of the summation undefined, as
different indices in equation~\eqref{eq:discrete_convolution} produce different
variants of discrete convolution, detailed in~\cref{tab:variant_equations}. The
\emph{linear convolution}, $\vcr y = \vcr f \ast \vcr g$, is equivalent to
equation~\eqref{eq:discrete_convolution} and using bounds that keep the indices
within the range of input and output vector dimensions. \emph{Cyclic
convolution} wraps the vectors by evaluating the indices modulo $n$.
Additionally, the inputs to cyclic convolution must be of equal size ($r=n$).
Equivalently, cyclic convolution is the linear convolution of a periodic signal
$\vcr g$. When we only want the subset of elements from linear convolution,
where every element of the filter is multiplied by an element of $\vcr g$, we
can use \emph{correlation} algorithms, as introduced by
Winograd~\cite{winograd}.  We can see these are the middle $n-r+1$ elements
from a discrete convolution.  Given a filter $\vcr f \in \mathbb{R}^r$ and
input $\vcr g \in \mathbb{R}^{n+r-1}$, correlation algorithms compute $n$
outputs.

\begin{table}[htb] 
\caption{Convolution variants: different formulae for the output element $y_i$ as well as 
the whole vector $\vcr y$.}
\label{tab:variant_equations} 
\centering 
{\footnotesize
\begin{tabular}{cccc} 
\noalign{\smallskip} \hline \hline
\noalign{\smallskip} 
    & Linear & Cyclic & Correlation\\
    \hline \\
    $\begin{medsize}y_k =\end{medsize}$ & 
        $\sum\limits_{i=\text{max}(0,k-n+1)}^{\text{min}(k,r-1)} f_ig_{k{-}i}$ & 
    $\sum\limits_{i=0}^{n-1} f_ig_{k-i (\text{mod } n)}$ &
    $\sum\limits_{i=0}^{r-1} f_ig_{k+i}$ \\
    $\vcr y = $ & 
    $\tpl fn \vcr g$ &
    $\mat{C}_{\langle \vcr f \rangle} \vcr g $ &
    $\begin{medsize}\begin{bmatrix} 
      f_0&\cdots &f_{r-1} & &\\
      &\ddots & &\ddots & \\
      & &f_{0}&\cdots&f_{r-1}\\
    \end{bmatrix}\end{medsize}\vcr g$  \\
    $\vcr y =$ & 
    $\tpl gr\vcr f$ &
    $\mat{C}_{\langle \vcr g \rangle} \vcr f $ &
    $\begin{bmatrix} 
        g_0 & \dots & g_{r-1}\\
        \vdots&&\vdots\\
        g_{n-1}&\dots&g_{n+r-2}
    \end{bmatrix}\vcr f$  \\ \\ \hline
\end{tabular} }
\end{table} 
Each of the three convolution algorithms can be used to solve the other two.
Using the Matrix Interchange Theorem (\cref{thm:mit}), we can derive
correlation algorithms from linear convolution algorithms and vice versa.
Cyclic convolution can be computed with linear convolution by appending inputs
$\vcr f$ and $\vcr g$ with copies of themselves. To compute linear convolution
with cyclic convolution, the inputs $\vcr f$ and $\vcr g$ are appended with
zeros until they are each of size $n+r-1$. 

Upon inspecting the summation of linear convolution, one can see each element
of $\vcr y$ is determined by an inner product of $\vcr f$ with a different
subset of $\vcr g$. This computation can be modeled as a matrix--vector
product $\tpl fn \vcr g$, where $\tpl fn \in \mathbb{R}^{(n+r-1) \times n}$ is
a lower--trapezoidal Toeplitz matrix with the structure,
\[
    \tpl fn = \begin{bmatrix} 
    f_0 & & \\
    \vdots &\ddots \\ 
    f_{r-1} & & f_0 \\ 
    &\ddots &\vdots \\ 
    & & f_{r-1} 
    \end{bmatrix}.
\]
Like linear convolution, a cyclic convolution is a
series of inner products between $\vcr f$ and different subsets of
$\vcr g$. The main difference is any summation past the last element
of $\vcr g$ ``wraps'' back to the start because of the
modular index, whereas in linear convolution the summation will
terminate. Consequently, we can denote cyclic convolution using
the matrix--vector product
$\mat{C}_{\langle \vcr f \rangle} \vcr g$, where $\mat{C}_{\langle
\vcr f \rangle} \in \mathbb{R}^{n \times n}$ is a circulant matrix
with the structure, 
\begin{align*}
    \mat{C}_{\langle \vcr f \rangle} = 
    \begin{bmatrix}
        f_0 & f_{n-1} & \cdots & f_1 \\
        f_1 & \ddots  & &\vdots  \\
        \vdots & & \ddots &  f_{n-1}\\
        f_{n-1} & \cdots &  f_1 & f_0
    \end{bmatrix}.
\end{align*}

Expressing convolution as a matrix-vector product with a structured matrix
allows the use of computational and analytical techniques for such matrices.
For example, given a square Toeplitz matrix of size $n$, the
determinant~\cite{pan2012structured}, as well as LU and QR
decompositions~\cite{bojanczyk1986qr} can be computed in $O(n^2)$ as compared
to the $O(n^3)$ needed for arbitrary
matrices~\cite{ye2016every,pan2012structured}.  Structured matrix inversion can
be computed in $O(n \log^2 (n))$ time~\cite{kumar1985fast,
bitmead1980asymptotically}.  Many of these fast algorithms, as we will see
later in this paper, can be derived by connections to polynomial algebra and
exploiting the structure of the computation.  We summarize various types of
convolution and ways to express them as products of a structured matrix and a
vector in~\cref{tab:variant_equations}.

Finally, we consider higher dimensional convolution methods.
Multidimensional convolution corresponds to convolving along each mode
of the inputs. Given a 2D filter $\mat F \in \mathbb{R}^{r \times r}$ and
input $\mat G \in \mathbb{R}^{n \times n}$, their linear convolution is
computed by 
\begin{equation}\label{eq:multi_convolution} 
    y_{lm} = 
    \sum\limits_{i = \text{max}(0,l-n+1)}^{\text{min}(l,r-1)}
    \sum\limits_{j = \text{max}(0,m-n+1)}^{\text{min}(m,r-1)}
    f_{ij} \cdot g_{l-i,m-j}.
\end{equation}

\subsection{Paper overview}
We survey different applications of convolution in~\cref{sec:applications}.
We formulate the convolution algorithm as a bilinear algorithm in~\cref{sec:general}, following Pan's formalism for 
matrix-multiplication~\cite{bilinear}. Then, we present
specific implementations of fast algorithms in~\cref{sec:fastImp}, \cref{sec:fastMod}, 
\cref{sec:fastAlg}, and~\cref{sec:adapt}.  We leverage our general formulation
of fast convolution algorithms to quantify their cost complexity
in~\cref{sec:costs}. We derive bounds on the numerical stability of
bilinear algorithms (providing a simplified summary of previous results
from~\cite{err1}) 
and provide solutions to reduce the error in~\cref{sec:stability}. We conduct
numerical experiments on the stability of a variety of 1D and multidimensional
convolution algorithms in~\cref{sec:experiments}. Finally, we present open
questions in~\cref{sec:future}.

\section{Problems and applications of convolution}\label{sec:applications} 
Convolution is a key component for many scientific and engineering
problems, such as signal processing, partial differential equations,
and image processing. In the following section, we examine how linear
discrete convolution and correlation convolution algorithms are used in
a variety of fields.

\subsection{Signal processing}\label{sec:signal_processing} 
One of the most important tasks in digital signal processing is the filtering
of a long signal, represented by a sequence of real and complex numbers.  The
filtering of the signal is calculated by a digital filter~\cite{survey3}, which
produces a new signal called an output sequence. FIR filters, or
finite-impulse-response filters, are digital filters that capture the strength
of the incoming signal for only a finite period of time~\cite{filter2}. The
computation of the output sequence from an FIR filter can be synthesized by
discrete convolution~\cite{gold1969direct}. The ubiquity of FIR filters within
domains such as noise removal in EKGs~\cite{filter1}, image processing for
texture mapping~\cite{texture}, and mobile communications~\cite{mobile} have
led to the development of highly efficient algorithms for 1D
discrete convolution, such as new nesting schemes~\cite{ifir,extending} and
the Fermat number transform~\cite{fnt}. 

\subsection{Integer multiplication} \label{sec:integer} 
Let $a$ and $b$ be two $n$-digit integers. The value of
the two integers can be rewritten as $a= \sum\limits_{i=0}^{n-1}
a_i \cdot 10^i$ and $b = \sum\limits_{i=0}^{n-1} b_i \cdot 10^i$, where
$a_i$ and $b_i$ are the individual digits for integers $a$ and $b$ respectively.
A direct computation of the product $a \times b$ can be formulated as
\begin{equation}
\label{eq:int_to_conv}
    a \times b = \sum\limits_{k=0}^{2n-2} 
    \sum\limits_{i=\text{max}(0,k-n+1)}^{\text{min}(k,n-1)} (a_i \cdot b_{k-i}) 10^k.
\end{equation}
The similarity of equation~\eqref{eq:int_to_conv} to
equation~\eqref{eq:discrete_convolution} allows integer multiplication to be
viewed as discrete convolution and vice versa.

\subsection{Numerical methods for differential equations} \label{sec:pde}
Within physics, chemistry, and engineering, many numerical PDE solvers are
based on determining the solution to continuous convolution equations.  For
example, integral equations for initial and boundary value problems of linear
differential equations seek to describe the solution $v$, which arises in
$u\ast v$ with the integration domain described by boundary conditions, where
$u$ is the Green's function of the differential
operator~\cite{kress1989linear}.

These problems are sometimes reduced to multidiscrete convolution, especially
when regular grids are used for discretization, often yielding 3D discrete
convolution problems. Iterative methods for solutions to convolution equations
leverage repeated application of the convolution operator, yielding a series of
discrete convolutions.  Techniques for fast convolution algorithms, such as the
discrete Fourier transform, also provide a way of solving convolution equations
directly.  Such regular-grid-based solvers are prevalent across a variety of
major numerical PDE applications in scientific computing.  For example, they
are used for acoustic scattering problems~\cite{BRUNO200180}, for long-range
interactions in molecular dynamics (particle-mesh Ewald
method)~\cite{darden1993particle}, and within quantum chemistry for electronic
structure
calculations~\cite{harrison2004multiresolution,khoromskij2008tensor,khoromskij2010fast}
and dynamics~\cite{pde1}.

Multidimensional convolution is a particularly important computational
primitive in methods for electronic structure calculations, which approximately
solve the many-body Schr\"odinger equation.  Standard formulations based on the
Hartree-Fock and Kohn-Sham equations, as well as Green's function
methods~\cite{fetter2012quantum}, involve multidimensional continuous
convolution. Solving these equations is generally done either using the
discrete multidimensional convolution~\cite{khoromskij2008tensor} or solving
them implicitly via Fourier transformations. Fast multidimensional convolution
algorithms (discussed in detail in Section~\ref{sec:lowrank}) have been
employed to approximately compute these convolutions with asymptotically less
cost~\cite{khoromskij2008tensor, khoromskij2010fast, khoromskij2009multigrid}.

\subsection{Convolutional neural networks} 
\label{sec:image_processing} 
Convolutional neural networks (CNNs) are a type of deep neural network that
uses low--level information, such as shapes and lines, to identify coarser
grained patterns, such as performing object recognition in image processing.
To gather local information, CNNs use many convolutions to compare subsets of
the data with a kernel (or filter). The success of CNNs in image
recognition~\cite{lenet,alex} catalyzed the recent rapid expansion of research
in deep learning~\cite{cnn1,cnn2,cnn3,cnn4}.  Convolution is the dominating
cost~\cite{sparse1} in CNNs, and thus it is desirable to improve its efficiency
to decrease both the training and inference time.

We now formally define how the series of 2D convolutions are performed for
image--processing based CNNs, which uses the \textit{correlation} form of
convolution. A CNN is associated with a set of $K$ filters of size $S \times R$
stored in the tensor $\tsr{F}$. An input to a CNN will be a set of $N$ images
stored in the tensor $\tsr{G}$. Each filter and image has $H$ channels, such as
the RGB channels for color images.  The convolutions are summed over the
channels and stored in $\tsr Y$,
\begin{equation} \label{eq:cnn_equation} 
    y_{ikxy} =
    \sum\limits_{c=1}^H \sum\limits_{v=1}^R \sum\limits_{u=1}^S
    f_{kcuv}\cdot g_{i,c,x+u,y+v}.
\end{equation}
In equation~\eqref{eq:cnn_equation}, the variable $i$ is the index for which of
the $N$ images we are convolving, and the variable $k$ denotes which of the $K$
filters is being used. Popular methods to efficiently
compute~\eqref{eq:cnn_equation} include casting the problem as a
matrix--multiplication~\cite{cudnn}, employing the FFT~\cite{fftw}, or
utilizing Winograd's minimal filtering method~\cite{lavin,winograd}.

The best choice of convolution algorithm depends on the parameters of the CNN
model. For CNNs with larger filters (approaching $r \geq 10$), the FFT is
competitive.  Currently, it is most popular to employ deeper (many layered)
CNNs with small filters ($r=2-4$)~\cite{you2019large}, in which case, standard
Winograd--based approaches work well. Even for slightly larger filter sizes,
different approaches become favorable, however.  For instance, commonly used
variants of the Winograd's convolution algorithm~\cite{lavin} suffer from
numerical instability due to large linear combination coefficients. As deep
learning research adopts approximate techniques such as
quantization~\cite{courbariaux2014training} (low--precision arithmetic), it
will be advantageous to have algorithms that balance numerical efficiency and
accuracy to prevent further perturbation to the computation. With the
proliferation of deep learning on GPUs and TPUs, it is desirable to cast
convolution in the language of linear algebra. Consequently, we present matrix
formulations for the various fast convolution algorithms. With these
formulations, our main goal is to illustrate the different algorithm choices
for computing convolution and to understand their complexity and numerical
accuracy.

\subsection{Multidimensional data analysis} \label{sec:cosmos}
In addition to CNNs, many other data--driven scientific discoveries rely on
convolution to better understand data from experimental observations and
computational simulations. As briefly listed in the beginning of the paper,
these applications include acoustics, geophysics, molecular simulation, and
quantum chemistry. Here, we provide some motivating examples of recent work
on application of CNNs in different scientific domains.

Within cosmology, the problem of learning parameters about
galaxies from distributions of matter previously relied on manually tuned
statistical measures using correlation functions~\cite{CosmoFlow, stv632}.
Recent experimental results~\cite{Mustafa2019, CosmoFlow} have shown that CNNs
can outperform the manually set measures, especially when analyzing noisy 
cosmological datasets~\cite{Schmelzle}. Despite the robustness of CNNs within
cosmology, the training time can take up to twenty days of run time when using
the TensorFlow framework~\cite{CosmoFlow}.

To improve performance, a CNN's parameters, such as its filter size, pooling
strategies, and stride length, must be  manually tuned or optimized over some
search space~\cite{elsken2018neural, balaprakash2019scalable}. Consequently,
the parameters of a CNN vary from application to application.  For example, in
past work on tumor classification, the manually constructed CNN contains a 1D
convolutional layer with 128 filters of kernel size
$r=10,20$~\cite{balaprakash2019scalable}, whereas cosmological parameter
estimation involves a 3D convolution with filters of size $r \times r \times r$
for $r=2,3,4$~\cite{CosmoFlow}.  The diversity of uses of CNNs in scientific
applications, including different dimensionality and filter sizes, motivates
exploration of general families of fast convolution algorithms to improve
performance over standard Winograd-based schemes and FFT.

\section{Bilinear algorithm representation for convolution} \label{sec:general}
A direct computation of a 1D linear convolution requires $O(nr)$ additions and
multiplications.  Faster algorithms can generally be represented using the
framework of bilinear algorithms~\cite{bilinear}. The linear convolution of 1D
vectors $\vcr f \in \mathbb{R}^r$ and $\vcr g \in \mathbb{R}^n$ can be defined
by a \emph{bilinear function},
\begin{equation} \label{eq:bi_func}
    \vcr y = \mathcal{F}_{\tsr T}(\vcr f,\vcr g), \text{ where } y_k = \sum\limits_{i,j}t_{ijk}f_ig_j, 
    \text{ with } t_{ijk}=\begin{cases} 1 : i+j-k=0 \\ 0 : \text{otherwise}\end{cases}.
\end{equation}
A CP decomposition~\cite{tensor1} of the tensor $\tsr T$, given by matrices 
$\mat A \in \mathbb{C}^{r \times R}$, $\mat B \in \mathbb{C}^{n \times R}$, and 
$\mat C \in \mathbb{C}^{(n+r-1) \times R}$ via
\begin{equation} \label{eq:bilinear_tensor} 
    t_{ijk} = \sum\limits_{l=0}^{R-1}a_{il}b_{jl}c_{kl}, 
\end{equation}
specifies a \emph{bilinear algorithm}~\cite{bilinear} for computing $\mathcal{F}_{\tsr T}$.
\begin{definition}[Bilinear algorithm]
    A bilinear algorithm $(\mat A, \mat B, \mat C)$ specifies an algorithm
    for computing $\vcr y = \mathcal{F}_{\tsr T}(\vcr f, \vcr g)$ via
    \begin{equation} 
    \label{eq:bilinear_summation} 
        y_k =
        \sum\limits_{l=0}^{R-1}c_{kl}\Big(\sum\limits_{i=0}^{r-1} a_{il}f_i\Big)
    \Big(\sum\limits_{j=0}^{n-1} b_{jl}g_j\Big), \text{ i.e., }  \ \  
    \vcr y = \mat C\bigg[(\mat A^\mathsf T\vcr f) \odot (\mat B^\mathsf T\vcr g) \bigg], \end{equation}
    where the value $R$ is the \emph{bilinear rank} of the algorithm.
\end{definition}
Matrices $\mat A$ and $\mat B$ specify linear combinations for inputs $\vcr f$
and $\vcr g$ respectively, which serve as respective inputs to a set of $R$
products of the two sets of linear combinations. The matrix $\mat C$ takes
linear combinations of these products to obtain each entry of the output $\vcr
y$. We refer to the multiplication by matrices $\mat A$ and $\mat B$ as
\emph{encoding} and the multiplication by matrix $\mat C$ as \emph{decoding}.
When an algorithm is applied recursively many times, the bilinear rank $R$
plays a key role, since the rank determines the number of recursive calls
needed. The asymptotic complexity of the recursive bilinear algorithm usually
depends on $R$ and not on the particular structure of the matrices $(\mat
A,\mat B,\mat C)$.

Given a filter of size $r$ and input of size $n$, a direct computation of
linear convolution has a rank of $R = nr$. Similarly, for filter size $r$ and
output size $n$, a direct computation of correlation convolution has bilinear
rank of $R = nr$. However, algorithms with bilinear rank of $n+r-1$ exist for
both problems. The optimality of this bilinear rank has been proven by
Winograd~\cite{winograd}.
\begin{theorem} \label{thm:winograd_minimal}
    The minimum rank of a correlation convolution algorithm with 
    filter of size $r$ and output of size $n$ is $n+r-1$.
\end{theorem}
A proof of the above theorem is presented in~\cite{winograd}. Winograd also
shows that by casting the bilinear algorithm to a trilinear algorithm, linear
convolution algorithms can be derived from correlation algorithms by swapping
variables, later defined as the \emph{matrix interchange}~\cite{err1,survey3}. 
We provide an alternative proof for the matrix interchange by simply swapping
the indices of the tensor $\tsr T$. 
\begin{theorem}[Matrix Interchange] \label{thm:mit}
    Let the bilinear algorithm for linear convolution $\vcr f$ and $\vcr g$ be defined
    as $\mat{C}\big( (\mat{A}^\mathsf T\vcr f) \odot (\mat{B}^\mathsf T\vcr g) \big)$. The correlation algorithm with
    output size $n$ is
    \begin{equation} \label{eq:matexchange}
        \mat B \Big( (\mat{A}^\mathsf T\vcr f) \odot (\mat{C}^\mathsf T\vcr g) \Big).
    \end{equation}
\end{theorem}
\begin{proof}
    From equation~\eqref{eq:bi_func}, the tensor $\tsr T$ in
    $\sum\limits_{ij}t_{ijk}f_ig_j$ satisfies $t_{ijk} = 1$ if $i+j-k=0$ and
    otherwise $t_{ijk}=0$. The bilinear function computing correlation can be
    expressed via tensor $\tsr T^{\text{corr}}$ as
    \[
        y_k = \sum\limits_{ij}t^{\text{corr}}_{ijk}f_ig_j=\sum\limits_{i=0}^{r-1}
        f_ig_{k+i}
    \] 
    with $t^{\text{corr}}_{ijk}=1$ if $i-j+k=0$, and consequently,
    \[
        t^{\text{corr}}_{ijk}=t_{ikj}.
    \] 
    Therefore, given a bilinear algorithm $(\mat A,\mat B,\mat C)$ to compute
    linear convolution, we obtain a bilinear algorithm $(\mat A,\mat C,\mat B)$
    for the correlation algorithm, since
    \[
        t^{\text{corr}}_{ijk}=t_{ikj}=\sum_{l=0}^{R-1}a_{il}b_{kl}c_{jl}.
    \]
\end{proof}
The conversion between correlation and linear convolution algorithm preserves
the number of element-wise multiplications, and subsequently the rank as well. 
\begin{corollary} \label{cor:min_rank}
    The minimum rank of a linear convolution with a filter of size $r$ and
    input of size $n$ is $R = n+r-1$.
\end{corollary}
We now present various bilinear algorithms that achieve the minimal rank.

\section{Convolution using polynomial interpolation} \label{sec:fastImp} 
Given a discrete set of points $\vcr x$ with corresponding values $\vcr y$,
interpolation derives the polynomial $\fun v$ the fits the values of $\vcr y$
as accurately as possible. Given $n$ points, a unique $n-1$ degree polynomial
$\fun v$ exists that satisfies $v(x_i) = y_i$ for all $i$~\cite{Heath}. 

Recall from~\cref{sec:integer} that polynomial multiplication is equivalent
to linear convolution. Let the vectors $\vcr f$ and $\vcr g$ be the coefficients for a
degree $r-1$ polynomial $\fun p$ and degree $n-1$ polynomial $\fun{q}$
respectively. The linear convolution of $\vcr f$ and $\vcr g$ is equivalent to
the coefficients of the polynomial product $\fun v = \fun p \fun q$. By viewing
linear convolution as polynomial multiplication, we can apply a family of fast
algorithms to convolution, one of which is based on interpolation. The
intuition behind the interpolation approach is as follows.
First, we multiply the values of $\fun{p}$ and $\fun{q}$ at
$n+r-1$ discrete nodes. These products are equivalent to $\fun{v}$ at those
same $n+r-1$ points. We then interpolate on these values to compute the
coefficients for polynomial $\fun{v}$. By carefully selecting the nodes and the
basis for interpolation, we can derive algorithms that are both stable and
compute linear convolution in asymptotically less time.

Let the matrix $\mat V \in \mathbb{C}^{R \times R}$ be the Vandermonde matrix with
$R = n+r-1$ distinct nodes. The bilinear algorithm's encoding matrices $\mat A
\in \mathbb{C}^{r \times R}$ and $\mat B \in \mathbb{C}^{n \times R}$ are
defined by keeping the first $r$ and $n$ rows of $\mat V^\mathsf T$,
respectively~\cite{winograd, survey3, opt}. The decoding matrix $\mat C \in
\mathbb{C}^{R \times R}$ is then given by $\mat V^{-1}$.  This construction of
the matrices $\mat{A}$, $\mat{B}$, and $\mat{C}$ creates a bilinear algorithm
with rank $R = n+r-1$.

\subsection{Karatsuba's Algorithm} 
In the late 1950s, Kolmogorov conjectured that integer 
multiplication~\eqref{eq:int_to_conv} had a cost complexity of $\Omega(n^2)$.
Karatsuba refuted the conjecture by developing an algorithm running in
$O(n^{\text{log}_2(3)})$ time~\cite{karatsuba1995complexity}. 
Karatsuba's algorithm reuses the previous element-wise multiplications
to compute the middle term of a two-digit integer multiplication problem,
\begin{align*}
\label{eq:kara_insight}
    a \times b &= \sum\limits_{k=0}^2
    \sum\limits_{i=\text{max}(0,k-1)}^{\text{min}(k,1)} (a_i \cdot b_{k-i})10^k\\
    &= (a_1 \cdot b_1)10^2 + (a_1\cdot b_0 + a_0 \cdot b_1)10 + (a_0 \cdot b_0)\\
    &= (a_1 \cdot b_1)10^2 + \big( (a_1\cdot b_1 + a_0\cdot b_0) -
    (a_0-a_1)(b_0-b_1)\big) 10 + (a_0 \cdot b_0).
\end{align*}
With the reformulation, the multiplication now only requires three unique
element-wise multiplications instead of four. When the inputs have more digits
than two, equation~\eqref{eq:kara_insight} can be applied by breaking the
integer into two smaller integers and recursively computing each element-wise
multiplication. By reducing the problem by a factor of two and making three
recursive calls, the asymptotic cost of this algorithm for $n$-digit integer
multiplication is $T(n) = 3T(n/2) + O(n) = O(n^{\log_2(3)})$. 

Karatsuba's algorithm operates in three distinct steps: take linear combinations
of the input, compute the element-wise multiplications, and compute the linear
combinations of the products. The combination of these three steps is captured
by the bilinear algorithm,
\begin{equation}
\label{eq:karatsubaLA}
    \begin{bmatrix}1&0&0\\1&-1&1\\0&0&1\end{bmatrix} \Bigg(
        \begin{bmatrix}1&0\\1&-1\\0&1\end{bmatrix}
        \begin{bmatrix}a_0\\a_1\end{bmatrix} \odot
        \begin{bmatrix}1&0\\1&-1\\0&1\end{bmatrix}
        \begin{bmatrix}b_0\\b_1\end{bmatrix} 
    \Bigg).
\end{equation}
This bilinear algorithm can be viewed as an interpolation-evaluation problem
using the nodes $0,1,\text{ and } \infty$. The use of the node $\infty$ will
be explained in the next section on Toom-Cook algorithms. Toom-Cook algorithms
encompass a family of fast algorithms, such as Karatsuba's algorithm, that
operate using a more general bilinear algorithm formulation. 

\subsection{The Toom-Cook method} 
\label{sec:toomcook}
Soon after the publication of Karatsuba's algorithm, Toom developed a
generalized algorithm for any input size $k$~\cite{toom1963complexity}. Cook's
Ph.D. thesis formalized Toom's algorithm into what is now known as the
Toom-Cook method~\cite{cook1969minimum}, which is an explicit definition of the
interpolation approach from the beginning of~\cref{sec:fastImp}. 

The designer of the Toom-Cook method can freely choose the basis and nodes.
Regardless of the basis, both the input and output must be represented in the
monomial basis, since convolution is equivalent to polynomial multiplication 
only in this basis. The Toom-Cook algorithm can be defined by the
Lagrangian basis~\cite{bodrato,zan,survey5,survey3}. Using this basis, the
polynomial multiplication, $\fun{v} = \fun{p}\fun{q}$, is computed by the summation,
\begin{equation} \label{eq:lagrange} 
    v(x) = \sum\limits_{j=0}^{r+n-2}\prod_{i = 0,i \ne j}^{r+n-2} 
    p(x_j) \cdot q(x_j) \frac{(x-x_i)}{(x_j-x_i)},  
\end{equation}
where $x_0,\ldots,x_{r+n-2}$ are the set of $r+n-1$ unique nodes. 
Equation~\eqref{eq:lagrange} can be rewritten as the multiplication by the
inverse Vandermonde matrix,
\begin{equation} \label{eq:vander_from_lagrange} 
    \vcr y = \mat V^{-1} \begin{bmatrix} p(x_0)\cdot q(x_0) \\ \vdots \\ 
    p(x_{r+n-2}) \cdot q(x_{r+n-2}) \end{bmatrix}.
\end{equation}
By defining the matrices $\mat A$ and $\mat B$ by the truncated Vandermonde matrix and
matrix $\mat{C}$ by the inverse Vandermonde matrix, as explained in the
beginning of~\cref{sec:fastImp}, the bilinear algorithm
$(\mat{A},\mat{B},\mat{C})$ computes the Toom-Cook algorithm.
A common choice of nodes are small integer values, such as 
$0,1,-1,2,-2,\ldots$. Small integers can limit the magnitude of
the scalars in the Vandermonde matrix. 

As the number of nodes increases, the number of non--zeros in the Vandermonde
matrix grows quadratically. The number of non--zeros in the $\mat{A}$,
$\mat{B}$, and $\mat{C}$ matrices can be reduced by selecting $\infty$ as a
node~\cite{survey5,opt}. The $\infty$ node computes the product between the
leading terms of inputs $\vcr f$ and $\vcr g$. To use the $\infty$ node,
the last row for each of the decoding matrices, $\mat A$ and $\mat B$, is set
to all zeros except for the last entry, which is set to $1$. Similarly, the
decoding matrix is set to $\mat C = \mat{\tilde{V}}^{-1}$, where
$\mat{\tilde{V}}$ is the original Vandermonde matrix with the last row set to
all zeros, and the last entry is set to $1$. The Karatsuba 
algorithm~\eqref{eq:karatsubaLA} is a Toom-Cook algorithm with the nodes
$0$,$1$, and $\infty$.

\subsection{Discrete Fourier transform} \label{sec:transforms} 
The use of integer nodes creates Vandermonde matrices that are ill-conditioned,
limiting the Toom-Cook method to small linear convolutions. Instead, the set of
nodes can be defined by the first $n$ non--negative powers of the primitive $n$th
primitive root of unity, $\omega_{(n)} = \text{exp}(-2\pi i/n)$. The use of the
powers of $\omega_{(n)}$ as nodes generates a Vandermonde matrix that is
equivalent to the discrete Fourier matrix, $\mat{D}^{(n)} \in \mathbb{C}^{n
\times n}$ with $d^{(n)}_{mk}=\omega_{(n)}^{mk}$.
The inverse of the discrete Fourier matrix is simply $\mat{D}^{(n)}{}^{-1} =
(1/n)\mat{D}^{(n)}{}^*$, so $\kappa(\mat{D}^{(n)})=1$. The ideal conditioning
of this matrix enables improved stability relative to Toom-Cook methods with
other choices of nodes. The use of the discrete Fourier matrix and its inverse
also defines bilinear algorithms for cyclic convolution~\cite{survey3}.
\begin{theorem}[Discrete cyclic convolution theorem] \label{thm:dct_conv_thm}
    The bilinear algorithm \\$(\mat{D}^{(n)}{}^\mathsf T, \mat{D}^{(n)}{}^\mathsf T, \mat D^{(n)}{}^{-1})$
    computes cyclic convolution.
\end{theorem}
\begin{proof}
    By expanding the bilinear algorithm,
    $\vcr y = {\mat{D}^{(n)}}^{-1} \big((\mat{D}^{(n)} \vcr f) \odot (\mat{D}^{(n)} \vcr g) \big)$,
    we have the summation,
    \begin{align*} 
        y_k &= \frac{1}{n} \sum\limits_{i=0}^{n-1}\omega_{(n)}^{-ki}
    \bigg(\sum\limits_{j=0}^{n-1}\omega_{(n)}^{ij}f_j\bigg)
    \bigg(\sum\limits_{t=0}^{n-1}\omega_{(n)}^{it}g_t\bigg)  
        = \frac{1}{n} \sum\limits_{i=0}^{n-1} \sum\limits_{j=0}^{n-1}
        \sum\limits_{t=0}^{n-1} \omega_{(n)}^{(j+t-k)i} f_jg_t.  
    \end{align*}
    It suffices to observe that for any fixed $u=j+t-k \ne 0$ or $\ne n$, the outer
    summation yields a zero result, since the geometric sum simplifies to
    \[
        \sum_{i=0}^{n-1} \omega_{(n)}^{ui} = (1-(\omega_{(n)}^{u})^n)/(1-\omega_{(n)}^{u}) = 0.
    \]
    Therefore the only non-zero values in the summation are $f_jg_{k-j \
    (\text{mod } n)}$, yielding cyclic convolution.
\end{proof} 
Recall that the cyclic convolution between $\vcr f$ and $\vcr g$ can be
computed as a circulant matrix--vector product, $\mat
{C}_{\langle \vcr f \rangle} \vcr g$. Then one can leverage the
eigendecomposition of the circulant matrix~\cite{bottcher2005spectral} to prove
the discrete cyclic convolution theorem.
\begin{proof}[Alternative proof of~\cref{thm:dct_conv_thm}]
    Using the eigendecomposition of the circulant matrix, $\mat {C}_{\langle
    \vcr f \rangle} = {\mat{D}^{(n)}}^{-1} \operatorname{diag}(\mat{D}^{(n)} \vcr f)
    \mat{D}^{(n)}$, and $\operatorname{diag}(\vcr a) \vcr b = \vcr a
    \odot \vcr b$ for vectors $\vcr a, \vcr b \in \mathbb{R}^n$, we can rewrite
    the matrix--vector product $\mat {C}_{\langle \vcr f \rangle} \vcr g =
    {\mat{D}^{(n)}}^{-1} \big( (\mat{D}^{(n)} \vcr f) \odot (\mat{D}^{(n)} \vcr g) \big)$.
\end{proof}

Other transformations to compute cyclic convolution may be defined based on
roots of unity in other finite fields.  One example is the Fermat number
transform (FNT)~\cite{fnt}. The FNT leverages roots of unity in the ring of
integers modulo the Fermat number, $F_{(n)} = 2^{2^t}+1$ for some non--negative
integer $t$. The roots of unity can then be selected as powers of $2$, yielding
a transformation that requires only $O(n\log(n))$ integer or bitmask additions
and bit-shifts.

\subsection{Fast Fourier transform} \label{sec:fft}
Applying the DFT using the fast Fourier transform (FFT) can reduce the
complexity of this algorithm from $O(n^2)$ to $O(n\log(n))$. The FFT
applies a divide-and-conquer structure to the DFT, which can be seen by
breaking the indices into even and odd components, 
\begin{equation}
\label{eq:hi} 
\begin{gathered} 
    y_k =
    \sum\limits_{i=0}^{n-1} x_i \omega_n^{ik}  =
    \sum\limits_{i=0}^{n/2-1}x_{2i}\omega_{n/2}^{ik} + \omega_{n}^k
    \sum\limits_{i=0}^{n/2-1} x_{2i+1} \omega_{n/2}^{ik}.
\end{gathered} 
\end{equation}
Computing both terms in equation~\eqref{eq:hi} recursively gives the
split-radix-2 variant of the Cooley-Tukey algorithm. In general, this division
can be extended to larger parities. For example, consider breaking an
$n=n_1n_2$-length FFT into $n_1$ FFTs of size $n_2$,
\begin{equation}
\label{eq:radix} 
    y_{(kn_1+t)}  =
    \sum\limits_{s=0}^{n_1-1}
    \omega_{n_1}^{st} \Bigg[\omega_{n}^{sk} \sum\limits^{n_2-1}_{i=0} x_{(i
    n_1+s)}\omega_{n_2}^{ik} \Bigg].
\end{equation}
This decomposition produces a split-radix-$n_1$ FFT algorithm,
which uses $n_1$ FFTs of size $n_2$ followed by $n_2$ FFTs of size $n_1$.
Both approaches yield an $O(n\log(n))$ cost.

\subsection{Discrete trigonometric transform} \label{subsec:dct}
A disadvantage of the DFT is its reliance of complex arithmetic. 
The discrete cosine transform (DCT) provides an alternative transformation that
is real-valued and preserves both the stability and the $O(n \log(n))$
complexity of the FFT. On the other hand, FFT-like algorithms for the DCT require
evaluation of trigonometric functions, and while usable for linear convolution,
the DCT requires a larger embedding (more padding with zeros) than with the
DFT. The DCT and its inverse correspond to evaluation and interpolation of a
polynomial in a Chebyshev basis. Consequently, the DCT is particularly useful
for multiplication of polynomials that are represented in a Chebyshev
basis~\cite{basz}, which also corresponds to the symmetric convolution of their
coefficients~\cite{dct,symdtt}.

The DCT of a vector $\vcr y \in \mathbb{R}^{N+1}$ is $\mat C_N^I \vcr y$, where
the matrix $\mat{C}_N^I \in \mathbb{R}^{(N+1) \times (N+1)}$ is defined as
\[
    [\mat C_N^I]_{ij} = \varepsilon_{N,j} \cdot \cos \Big( \frac{i \cdot j \cdot \pi }{N} \Big) 
	\text{ s.t. } 
	\varepsilon_{N,j} = 
	\begin{cases} \frac{1}{2} \ : \ $j=0,N$ \\ 1 \ : \ \text{otherwise}
	\end{cases}.
\]
The superscript $I$ signifies that this is a DCT-$1$ transform.  The different
DCT types, ranging from the DCT-$1$ to DCT-$4$, differ by the shifts to $i$ and
$j$ inside the cosine function used to construct the basis and in the
definition of the first row/column of the matrix~\cite{dct}. Further, $\mat
C_N^I{}^{-1} = \frac{2}{N}\mat C_N^I$~\cite{basz} and the DCT is essentially
ideally conditioned. Linear convolution of $n$-dimensional vectors can be
computed via the DCT by first pre-padding with $\lfloor n/2 \rfloor + 1$ zeros
and post-padding with $\lfloor 3n/2 \rfloor+2$ zeros to both input vectors. Let
$\hat{\vcr{y}}$ be the output from the DCT-$1$ based bilinear algorithm with
the two padded vectors as inputs. Then the solution to linear convolution is
embedded in $2 \hat{\vcr{y}}$ from indices $2\lfloor n/2 \rfloor + 3$ to $2
\lfloor 3n/2 \rfloor + 1$, inclusively. However, for linear convolution, the
need to perform padding and the cost of evaluating trigonometric functions
generally makes DFT-based methods preferable to DCT.

\section{Convolution using modular polynomial arithmetic} \label{sec:fastMod}
Winograd presents a more general family of convolution
algorithms~\cite{winograd} based on modular arithmetic over polynomials.
Consider evaluating the remainder of the product $\fun{v} = \fun{p}\fun{q}$,
expressed as $\fun{\rho} = \fun{v} \text{ mod } \fun{M}$.  When $\deg{\fun{M}}
> \deg{\fun{v}}$, where we denote the degree of a polynomial by $\deg{\cdot}$,
then $\fun{\rho} = \fun{v}$.  If instead $\deg{\fun{M}} \le \deg{\fun{v}}$,
then $\fun{\rho} \ne \fun{v}$, as the remainder of $\fun{v}/\fun{M}$ will
produce a polynomial $\fun{\rho}$ of degree at most $\deg{\fun{M}}-1$.
Winograd shows that computing remainders of $\fun{v}$ (evaluating $\fun{p}$ and
$\fun{q}$) with well--chosen polynomial divisors will produce new fast and
stable linear convolution algorithms. We first present Winograd's algorithm for
recovering $\fun{v}$ with $\deg{\fun{M}} > \deg{\fun{v}}$.

\subsection{Winograd's convolution method}
In interpolation, each polynomial is evaluated at a set of discrete 
points. In Winograd's convolution algorithm, the remainder of the
product $\fun v = \fun p \fun q$ is computed using $k$ distinct
polynomial divisors, $\fun{m^{(i)}}$. The $k$ polynomial divisors,
$\fun{m^{(1)}}, \fun{m^{(2)}}, \cdots, \fun{m^{(k)}}$,
must be coprime, or share no common roots. Together, the product
of the $k$ polynomials define the larger polynomial divisor, 
$\fun{M} = \prod_i \fun{m^{(i)}}$. After computing the remainders with each
the $k$ polynomial divisors, $m^{(i)}$, the remainder $\fun{\rho} = \fun{v} \text{
mod } \fun M$ is recovered via the Chinese remainder theorem.

The \emph{Chinese remainder theorem} for polynomials provides a specification
for recovering the product $\fun{v} = \fun{p}\fun{q} \pmod{\fun{M}}$ from the set of
$k$ polynomial remainders of $\fun v$,
\begin{equation} \label{eq:winembed}
    \fun{u^{(i)}} \equiv \fun{v} \pmod{\fun{m^{(i)}}}.
\end{equation}
The bound on degree, in
combination with the fact that $\fun{m^{(1)}}, \ldots, \fun{m^{(k)}}$ are
coprime, ensures that the remainder polynomials $\fun{u^{(1)}},\ldots,
\fun{u^{(k)}}$ uniquely specify $\fun{v}$. Consequently, defining
$\fun{M^{(i)}} = \fun{M}/\fun{m^{(i)}}$, B\'{e}zout's identity implies that
there exists polynomials $\fun{n^{(i)}}$ and $\fun{N^{(i)}}$ such that
\begin{equation} \label{eq:bezout_eq}
    \fun{M^{(i)}}\fun{N^{(i)}} + \fun{m^{(i)}}\fun{n^{(i)}} = 1.
\end{equation}
A set of such polynomials $\fun{N^{(1)}},\cdots, \fun{N^{(k)}}$ can be computed
by the extended Euclidean algorithm. Later in~\cref{lem:e_mat}, we provide a
numerical formulation for the extended Euclidean algorithm. The desired
polynomial $\fun{v}$ satisfying the set of equivalences~\eqref{eq:winembed} can
be recovered as
\begin{align} \label{eq:berz}
    \fun{v}  = \Big(\sum_{i=1}^{k} \fun{u^{(i)}} \fun{M^{(i)}} \fun{N^{(i)}}\Big) \bmod{\fun{M}},
\end{align}
since $\fun{u^{(i)}} \fun{M^{(i)}} \fun{N^{(i)}} \equiv 0 \pmod{\fun{m^{(j)}}}$
for $i\neq j$, while 
\[
    \fun{u^{(i)}} \fun{M^{(i)}} \fun{N^{(i)}} 
    = \fun{u^{(i)}}(1-\fun{m^{(i)}}\fun{n^{(i)}}) 
    \equiv \fun{u^{(i)}} \pmod{\fun{m^{(i)}}}.
\]

Interpolation is a particular instance of a Winograd's convolution algorithm.
By selecting the polynomial divisors $\fun{m^{(i)}}$ to be the polynomial
$x-\chi_i$, where $\chi_i$ are nodes, Winograd's algorithm is equivalent to the
Toom-Cook method using Lagrangian interpolation~\cite{survey3}.  The DFT
algorithm for linear convolution may be obtained by the polynomial $M(x) =
x^{k} -1$ with $k=n+r-1$, whose roots are equally spaced on the unit circle on
the complex plane~\cite{survey5}.  With the choice $M(x) = x^{k} -1$ for
$k=n=r$, we obtain cyclic convolution~\cite{survey7}, since the remainder
polynomial $\fun{\rho}$ has the right coefficients, namely \[
    \sum_{i=0}^{2n-1}v_ix^i 
    \equiv \underbrace{\sum_{i=0}^{n-1}(v_{n+i}+v_i)x^i}_{\rho(x)} \pmod{x^{n}-1}.
\]
The polynomial divisors $\fun{m^{(i)}}$ can also be chosen to be of degree $d >
1$ (superlinear polynomials)~\cite{beyond}. Different degree choices for the
polynomial divisors will yield trade-offs between the bilinear rank and the
number of additions necessary.  A few examples of this trade-off are shown in
\cref{tab:costTb1}~\cite[Table~5.2]{survey3}.  The degree choices also affect
numerical stability.
\begin{table}[htb] 
\caption{Number of additions for Winograd's convolution algorithm with different bilinear ranks} 
\centering 
\label{tab:costTb1} 
\begin{tabular}{rccc} 
\noalign{\smallskip} \hline \hline
\noalign{\smallskip} 
    $n$ & $r$ & Rank & Adds \\ 
    \hline 
    2 & 2 & 3 & 3 \\
    2 & 2 & 4 & 7 \\
    3 & 3 & 5 & 20 \\
    3 & 3 & 6 & 10 \\
    3 & 3 & 9 & 4 \\
    4 & 4 & 7 & 41 \\
    4 & 4 & 9 & 15
\\ \noalign{\smallskip} \hline
\noalign{\smallskip} 
\end{tabular} 
\end{table} 

\subsection{Bilinear algorithm for Winograd's convolution method} \label{sec:winoalg_bilinear}
We now present a formulation of the bilinear algorithm for Winograd's
convolution algorithm.  As before, we denote the coefficients of an arbitrary
polynomial $\fun{p}$ as $\vcr p$. Let $\xmm md \in \mathbb{C}^{\deg{\vcr m}
\times (d+1)}$ be a matrix that can act on the coefficients of any degree $d$
polynomial $\fun p$ to compute the coefficients of $\rho = \fun p \pmod{m}$ as
$\vcr{\rho} = \xmm md \vcr p$ as proposed in~\cite{survey5}.  We provide a succinct
algebraic construction of this linear operator,
\begin{equation}
    \xmm md = 
    \begin{bmatrix}
        \mat I & - \mat L\mat{U}^{-1}
    \end{bmatrix},
\end{equation}
where $\mat{I}$ is an identity matrix of size $\deg{m}$, $\mat{L}$ contains the
top $\deg{\fun m}$ rows of $\tpl m{d-\deg{m}+1}$, and $\mat{U}$ contains the
bottom $d+1$ rows of $\tpl m{d-\deg{m}+1}$.
\begin{lemma}
    Let $\fun{\rho} = \fun{p} \pmod{\fun{m}}$, with $d=\deg{p}$, then
    \(\vcr{\rho} = \xmm md \vcr{p}.\)
\end{lemma}
\begin{proof}
    Let $\fun{q} = \fun{p}/\fun{m}$, so that $\fun{\rho} = \fun{p} - \fun{q}\fun{m}$.
    As $\deg{\fun{\rho}} < \deg{\fun p}$, then $\deg{\fun p} = d-\deg{\fun m}$.
    Defining $w=qm$, let 
    \[
        \vcr{p} = \begin{bmatrix}\vcr{p}_\text{upper} \\ 
        \vcr{p}_\text{lower} \end{bmatrix} \text{ and } 
        \vcr{w} = \begin{bmatrix}\vcr{w}_\text{upper} \\ 
        \vcr{w}_\text{lower}\end{bmatrix},
    \]
    where $\vcr{p}_\text{upper}, \vcr{w}_\text{upper}\in\mathbb{C}^{\deg{\fun{m}}}$, 
    so $\vcr{p}_\text{upper} = \begin{bmatrix} \mat{I} & \mat{O} \end{bmatrix} \vcr{p}$.
    Then we have that $\vcr{\rho} = \vcr{p}_\text{upper} - \vcr{w}_\text{upper}$.
    Furthermore, observing that $\vcr{w} = \tpl m{\deg{\fun q}+1}\vcr q$ and
    separating 
    $\tpl m{\deg{\fun p}+1} = \begin{bmatrix} \mat{L} \\ \mat{U}\end{bmatrix}$,
    where $\mat{L}\in\mathbb{C}^{\deg{m}\times (d-\deg{m}+1)}$ is lower-triangular and 
    $\mat{U}\in\mathbb{C}^{(d-\deg{m}+1)\times(d-\deg{m}+1)}$ is upper-triangular, we have
    \[
        \vcr w_\text{upper} = \mat{L}\mat{q}.
    \]
    Further, since 
    $\vcr{w} = \Big(\vcr{p} - \begin{bmatrix} \vcr{\rho} \\ \vcr 0\end{bmatrix}\Big)$,
    we have that $\vcr{p}_\text{lower} = \vcr{w}_\text{lower} = \mat{U}\mat{q}$,
    and so $\mat{q}=\mat{U}\vcr{p}_\text{lower}$.
    Therefore, we obtain
    \[
        \vcr{\rho} 
        = \vcr{p}_\text{upper} - \mat{L}\mat{U}^{-1}\vcr{p}_\text{lower} 
        = \xmm md \vcr{p}.
    \]
\end{proof}
Using this linear operator, we can now construct an operator for modular
polynomial multiplication. Since,
\[
    \fun p \fun q \bmod m =  (\fun p \bmod m)(\fun q \bmod m) \bmod m,
\]
we have that
\[
    \xmm m{\deg p + \deg q - 1} (\vcr p \ast \vcr q) 
    = \xmm m{2\deg m - 1}\big((\xmm m{\deg p}\vcr p)\ast (\xmm m{\deg q}\vcr q) \big).
\]
Further, given a bilinear algorithm $(\mat A, \mat B, \mat C)$ to compute
linear convolution of two $m$-dimensional vectors, we can obtain an algorithm
to compute $\fun{\rho} = \fun p \fun q \bmod m$,
\[
    \vcr{\rho} =\xmm m{2\deg m - 1}\mat 
    C \big((\mat A^\mathsf T \xmm m{\deg p} \vcr p) \odot (\mat B^\mathsf T \xmm m{\deg q}\vcr q)\big).
\]
To implement the Winograd's convolution algorithm, we need to compute $\vcr p
\vcr q \bmod m^{(i)}$ for $i\in \{1,\ldots, k\}$ to obtain the coefficients of
$\vcr u^{(1)}, \ldots, \vcr{u}^{(k)}$ in~\eqref{eq:winembed}.  After obtaining
these remainders $\vcr u^{(1)}, \ldots, \vcr{u}^{(k)}$, it suffices to
compute~\eqref{eq:berz} by multiplying each $\vcr u^{(i)}$ with the matrix,
\[
    \xmm M{\deg{M}+\deg{m^{(i)}}-2}\tpl{e^{(i)}}{\deg{m^{(i)}}} \xmm {m^{(i)}}{2\deg{m^{(i)}}-1}, 
\]
where $e^{(i)} = M^{(i)}N^{(i)}\bmod{\fun{M}}$. Consequently, we can interpret
Winograd's convolution algorithm as a prescription for building a new bilinear
algorithm for convolution from a set of $k$ bilinear algorithms that compute
the linear convolution between two sequences of vectors with dimension
$\deg{m^{(1)}},\ldots, \deg{m^{(k)}}$.
\begin{theorem}[Winograd's Convolution Algorithm]
    Given $\fun{M}=\prod_{i=1}^km^{(i)}$ where $\deg{M} = n+r-1$ and
    $m^{(1)},\cdots, m^{(k)}$ are coprime, as well as $(\mat A^{(i)},\mat
    B^{(i)},\mat C^{(i)})$ for $i\in\{1,\ldots, k\}$, where $(\mat A^{(i)},\mat
    B^{(i)},\mat C^{(i)})$ is a bilinear algorithm for linear convolution of
    vectors of dimension $\deg{m^{(i)}}$, Winograd's convolution algorithm
    yields a bilinear algorithm $(\mat A, \mat B, \mat C)$ for computing linear
    convolution with vectors of dimension $r$ and $n$, where
    \begin{align*}
            \mat A &= \begin{bmatrix}
                \xmm{m^{(1)}}{r-1}^\mathsf T\mat A^{(1)} & \cdots & 
                                     \xmm{m^{(k)}}{r-1}^\mathsf T\mat A^{(k)})
            \end{bmatrix}, \\
            \mat B &= \begin{bmatrix}\xmm{m^{(1)}}{n-1}^\mathsf T\mat B^{(1)} & \cdots & 
                                     \xmm{m^{(k)}}{n-1}^\mathsf T\mat B^{(k)})
            \end{bmatrix}, \text{and} \\
            \mat C &= \begin{bmatrix}
                \mat{\tilde{C}}^{(1)} & \cdots & \mat{\tilde{C}}^{(k)}
            \end{bmatrix},
    \end{align*}
    with $\mat{\tilde{C}}^{(i)} = 
        \xmm{M}{\deg{M}+\deg{m^{(i)}}-2}\tpl{e^{(i)}}{\deg{m^{(i)}}}
        \xmm {m^{(i)}}{2\deg{m^{(i)}}-1}\mat C^{(i)}$
    and polynomial $e^{(i)} = M^{(i)}N^{(i)}\bmod{\fun{M}}$. 
\end{theorem}
To automatically generate Winograd's convolution algorithm, it suffices to have
a prescription to obtain $e^{(i)} = M^{(i)}N^{(i)}\bmod{\fun{M}}$.  Below, we
present a matrix formulation for solving B\'{e}zout's identity, which is
similar to computing a polynomial division via a triangular--Toeplitz linear
system of equations~\cite{pan2012structured}.
\begin{lemma} \label{lem:e_mat}
    Given coprime polynomials $\fun{\hat{M}}$ and $\fun{\hat{m}}$, the coefficients of 
    polynomials $\fun{\hat{N}}$ and $\fun{\hat{n}}$ satisfying
    $\fun{\hat{M}} \fun{\hat{N}} + \fun{\hat{m}}\fun{\hat{n}} = 1$ are
    \begin{equation}
        \begin{bmatrix}\vcr{\hat{N}} \\ \vcr{\hat{n}} \end{bmatrix} = 
        \begin{bmatrix}
            \tpl{\hat M}{\deg{\hat{m}}-1} & \tpl{\hat{m}}{\deg{\hat{M}}-1}
        \end{bmatrix}^{-1}
        \begin{bmatrix}1 \\ 0 \\ \vdots \\ 0\end{bmatrix}.
    \end{equation}
\end{lemma}
\begin{proof}
    The polynomials degrees of $\fun{\hat{N}}$ and $\fun{\hat{n}}$ are at most
    $\deg{\fun{\hat{N}}} \le \deg{\fun{\hat{m}}}-1$ and 
    $\deg{\fun{\hat{n}}} \le \deg{\fun{\hat{M}}}-1$~\cite{annala2016bezout}.
    Therefore, we can rewrite the equivalence 
    $\fun{\hat{M}} \fun{\hat{N}} + \fun{\hat{m}}\fun{\hat{n}} = 1$ as
    \begin{equation} \label{eq:ide_sys}
           \underbrace{\begin{bmatrix}
               \tpl{\hat{M}}{\deg{\hat{m}}-1} & \tpl{\hat{m}}{\deg{\hat{M}}-1}
           \end{bmatrix}}_{\mat{A}}
                \begin{bmatrix}\vcr{\hat{N}}\\\vcr{\hat{n}}\end{bmatrix}
            = \begin{bmatrix}1\\0\\ \vdots \\ 0 \end{bmatrix}.
    \end{equation}
    To show that the matrix $\mat{A}$ is invertible, we demonstrate that there
    cannot exist a vector $\vcr x \in \mathbb{C}^{\deg{\hat{m}} +\deg{\hat{M}}}$,  $\vcr x
    \neq \vcr 0$, such that $\mat A\vcr x = \vcr 0$. Equivalently, we show 
    there cannot exists vectors $\vcr{\hat{N}}$ and $\vcr{\hat{n}}$ such that
    $\tpl{\vcr{\hat{M}}}{\deg{\hat{m}}-1}\vcr{\hat{N}} = -\tpl{\vcr{\hat{m}}}{\deg{\hat{M}}-1}\vcr{\hat{n}}$. Since
    $\fun{\hat{M}}$ and $\fun{\hat{m}}$ are coprime, $\fun{\hat{N}}$ must be a multiple of $\fun{\hat{m}}$.
    However, because $\deg{\fun{\hat{N}}} < \deg{\fun{\hat{m}}}$, there cannot exist such a
    polynomial $\fun{\hat{N}}$.
\end{proof}

\section{Other fast algorithms for convolution} \label{sec:fastAlg} 
We now discuss two other techniques for fast convolution, which are not based
on polynomial algebra.

\subsection{Fast symmetric multiplication} \label{sec:symImp}
Recall that convolution can be solved by a Toeplitz matrix--vector product,
$\vcr y = \tpl fn\vcr g$.  Consider, for simplicity, the scenario when $n=r$ is
the dimension of both $\vcr f$ and $\vcr g$.  This problem can be converted to
a Hankel matrix-vector product by reversing the order of the elements in the
vector $\vcr{g}$ with $\vcr{y} = \hkl {f}n\vcr{\hat{g}}$, where
\[
    \hkl{f}n =
    \begin{bmatrix}
        &        & f_1 \\
        &   \adots & \vdots \\
        f_1   & & f_n \\
        \vdots & \adots  & \\
        f_n &   &   \\
    \end{bmatrix}.
\]
We can embed (for simplicity) $\hkl{f}n$ within a square Hankel matrix,
$\mat{H}_{(2n-1)}(\vcr{x})\in\mathbb{R}^{(2n-1)\times(2n-1)}$, by appending
$n-1$ zero columns to $\hkl{f}n$ (using 
$\vcr x = \begin{bmatrix} \vcr 0^\mathsf T & \vcr{f}^\mathsf T & \vcr{0}^\mathsf T\end{bmatrix}^\mathsf T$ 
to define each anti-diagonal of the matrix), so that 
$\vcr{y} = \mat H_{(2n-1)}(\vcr{x})\begin{bmatrix} \vcr{\hat{g}} \\ 
\vcr 0 \end{bmatrix}$.
Now, we can observe that this square Hankel matrix is symmetric, and further
that this type of matrix can be subdivided recursively into Hankel matrices,
\[
    \mat{H}_{(kl)}(\vcr{x}) 
    = \begin{bmatrix} 
        \mat{H}_{(k)}(\vcr{x}_1) &\cdots & \mat{H}_{(k)}(\vcr{x}_l) \\ 
        \vdots & & \vdots \\ 
        \mat{H}_{(k)}(\vcr{x}_l) & \cdots & \mat{H}_{(k)}(\vcr{x}_{2l-1})
    \end{bmatrix}.
\]
Consequently, we can leverage fast nested bilinear algorithms to compute the
product of a symmetric matrix and a vector~\cite{SD_ETHZ_2015}.  These
algorithms compute the multiplication of an $l\times l$ symmetric matrix with a
vector using $l(l+1)/2$ multiplications.  The choice of $l=2$, requires 3
multiplications, and yields the fastest asymptotic complexity (same as
Karatsuba's algorithm $O(n^{\log_2(3)})$).  This variant of the algorithm
performs the Hankel matrix--vector product
$\vcr{y} = \mat{H}_{(2k)}(\vcr{x})\vcr{z}$ using the transformation,
\begin{align*}
  \begin{bmatrix} \vcr{y}_1 \\ \vcr{y}_2 \end{bmatrix} &= 
  \begin{bmatrix} 
  \mat{H}_{(k)}(\vcr{x}_1)\vcr{z}_1  +  \mat{H}_{(k)}(\vcr{x}_2)\vcr{z}_2 \\
  \mat{H}_{(k)}(\vcr{x}_2)\vcr{z}_1  +  \mat{H}_{(k)}(\vcr{x}_3)\vcr{z}_2
  \end{bmatrix} \\
  &=
  \begin{bmatrix} 
  \big(\mat{H}_{(k)}(\vcr{x}_1)-\mat{H}_{(k)}(\vcr{x}_2)\big)\vcr{z}_1  +  \mat{H}_{(k)}(\vcr{x}_2)(\vcr{z}_1 + \vcr{z}_2) \\
  \mat{H}_{(k)}(\vcr{x}_2)(\vcr{z}_1 + \vcr{z}_2)  +  \big(\mat{H}_{(k)}(\vcr{x}_3)-\mat{H}_{(k)}(\vcr{x}_2)\big)\vcr{z}_2
  \end{bmatrix}.
\end{align*}
The new form can be computed with $3$ Hankel--vector products of half the
dimension. The addition of the Hankel submatrices can be computed with $O(r)$
additions. Therefore, the cost of the fast symmetric algorithm is $T(n) =
3T(n/2) + O(n) = O(nr^{\log_2(3/2)})$ by directly computing the convolution
once $n \approx r$.

\subsection{Minimizing scalar products} \label{sec:sparse_bilinear}
There remain other bilinear algorithms for convolution not covered by the
techniques in the previous sections. For example, a bilinear algorithm for
linear convolution of $3$-dimensional vectors can be derived by the
factorization \cite{survey3},
\begin{equation}
    \begin{bmatrix}
        1&0&0&0&0&0\\
        -1&-1&0&1&0&0\\
        -1&1&-1&0&1&0\\
        0&-1&-1&0&0&1\\
        0&0&1&0&0&0
    \end{bmatrix}
    \left(
        \begin{bmatrix}
            1&0&0\\
            0&1&0\\
            0&0&1\\
            1&1&0\\
            1&0&1\\
            0&1&1
        \end{bmatrix}
        \begin{bmatrix}
            f_0\\f_1\\f_2
        \end{bmatrix}
        \odot
        \begin{bmatrix}
            1&0&0\\
            0&1&0\\
            0&0&1\\
            1&0&0\\
            1&0&1\\
            0&1&1
        \end{bmatrix}
        \begin{bmatrix}
            g_0\\g_1\\g_2
        \end{bmatrix}
    \right).
\end{equation} 
While this bilinear algorithm does not achieve the minimal rank,
the cost of encoding and decoding is lower than for the bilinear algorithm of
the optimal rank since $(\mat A, \mat B, \mat C)$ are sparse and require only
additions or subtractions to apply.

\section{Adaptations of convolution algorithms} \label{sec:adapt}
All fast algorithms described so far can be adapted to efficiently perform
convolution when the filter size is small, i.e., $r\ll n$, and can be applied to
multidimensional convolution. We describe multidimensional convolution
adaptations using the bilinear algorithm representation.

\subsection{Convolution with small filters} \label{sec:small_filter}
Many popular CNN architectures today use filters (referred to as kernels in
CNNs) that are small in size. The 2D filter's size ranges from $11 \times 11$
down to $3 \times 3$ \cite{cnn1,cnn2,cnn3,cnn4,alex}, whereas the images are of
dimension $256 \times 256$ and larger \cite{alex}. While fast algorithms, such
as the interpolation approach, can produce efficient convolution algorithms for
any input size, these algorithms are subject to large errors when the
dimensions are larger than four \cite{lavin, winograd}.

The cost and/or error of convolution can be reduced by breaking a long
convolution into a series of smaller convolutions. One simple approach is to
divide a vector into a series of small vectors. As an example, consider a
Toeplitz matrix--vector multiplication, $\tpl{f}{n}\vcr g$, for computing
the 1D linear convolution $\vcr f \ast \vcr g$, where $r \ll n$. We can
represent the products in its block form,
\begin{equation} \label{eq:blockToe} 
    \tpl fn\vcr g = 
    \begin{bmatrix} 
        \mat{A} & && \\ 
        \mat{B} & \mat{A} & &\\ 
         &   \ddots &  \ddots & \\ 
         &  & \mat{B} & \mat{A}
    \end{bmatrix} 
    \begin{bmatrix} \vcr{g_{n/r-1}} \\ \vdots \\
        \vdots  \\ \vcr{g_0}
    \end{bmatrix}.
\end{equation}
The block Toeplitz matrix $\tpl fn$ can be written using Kronecker
products~\cite{tensorML}, $\tpl fn = (\mat{I} \otimes \mat{A}) + (\mat{I'}
\otimes \mat{B})$, where $\mat{I'}$ is a matrix with a sub-diagonal of ones.
Let $\mat{G} \in \mathbb{R}^{r \times (n/r)}$ be the matrix where
$\text{vec}(\mat{G}) = \vcr g$.  We can rewrite the Toeplitz matrix--vector
multiplication problem as
\begin{align*} \label{eq:tensor_small_filter} 
    \tpl fn \vcr g &= (\mat{I}\otimes \mat{A})\text{vec}(\mat G) 
    + (\mat{I'} \otimes \mat{B})\text{vec}(\vcr G) \\
    &= \text{vec} (\mat{A} \mat G) + \text{vec}(\mat{B} \mat G \mat{I'}{}^\mathsf{T}).
\end{align*}
Given a fast convolution algorithm with a cost of $T(r)$ when both inputs
are size $r$, the asymptotic complexity of computing this entire convolution is
$O\big(\frac{n}{r} \cdot T(r)\big)$. When $n \gg r$, this formulation can
reduce the cost of the fast Fourier transform from $O(n\text{log}(n))$ to
$O(n\text{log}(r))$.

\subsection{Multidimensional convolution via 1D convolution} \label{sec:multidimensional}
For problems in image processing and scientific computing, where the inputs are
2D, 3D, or 4D, we need methods for multidimensional convolution. We provide a
way to construct 2D convolution algorithms from 1D convolution algorithms,
which extends in a natural way to higher-dimensional convolutions.  Given
$\mat{F}\in\mathbb{R}^{r\times r}$ and $\mat{G} \in \mathbb{R}^{n\times n}$,
the 2D linear convolution $\mat{Y} = \mat{F} \ast \mat{G}$ with
$\mat{Y}\in\mathbb{R}^{(n+r-1)\times(n+r-1)}$ gives
\begin{equation} \label{eq:nd_bilinear} 
\begin{gathered} 
y_{ab} =
    \sum_{i=\max(0,a-n+1)}^{\min(a,r-1)} 
    \sum_{j=\max(0,b-n+1)}^{\min(b,r-1)} f_{ij}g_{a-i,b-j}.
\end{gathered} 
\end{equation}
A 2D convolution can be broken into a convolution of convolutions. That is,
each row is individually convolved and then the rows are convolved amongst each
other. Given a bilinear algorithm for a linear 1D convolution, $(\mat A,\mat
B,\mat C)$, the bilinear algorithm for a linear 2D convolution~\cite{lavin} is
\begin{equation} \label{eq:lavin2D} 
    \mat Y = \mat C \Big[ (\mat A^\mathsf T\mat F\mat A) 
            \odot (\mat B^\mathsf T\mat G\mat B) \Big] \mat C^\mathsf T.
\end{equation}
Correctness of this algorithm can be shown by defining the 2D convolution
tensor, $\tsr{T}^{(2D)} = \tsr{T}\otimes \tsr{T}$, so that
\[
    t^{(2D)}_{in+j,ur+v,a(n+r-1)+b} = t_{iua}t_{jvb}.
\] 
This tensor computes 2D convolution as $\hat{y}_{k} = \sum_{i,j}
t^{(2D)}_{ijk}\hat{f}_i\hat{g}_j$, where
$\vcr{\hat{y}}=\text{vec}(\mat{Y})$,
$\vcr{\hat{f}}=\text{vec}(\mat{F})$, and
$\vcr{\hat{g}}=\text{vec}(\mat{G})$,
since,
\[
    y_{ab} 
    = \sum_{i=0}^{r} \sum_{j=0}^{r} 
    \sum_{u=0}^{n} \sum_{v=0}^{n} t_{iua}t_{jvb}f_{ij}g_{uv}.
\]
A rank $R^2$ decomposition of $\tsr{T}^{(2D)}$ can be constructed from a rank
$R$ decomposition of $\tsr{T}$ as $(\mat{A} \otimes \mat{A}, \mat{B}\otimes
\mat{B}, \mat{C}\otimes \mat{C})$.  The resulting bilinear algorithm,
\[
    \vcr{\hat{y}} 
    = (\mat{C}\otimes \mat{C}) \Big[ 
      \big((\mat{A}\otimes \mat{A})^\mathsf T \vcr{\hat{f}}\big) \odot 
      \big((\mat{B}\otimes \mat{B})^\mathsf T \vcr{\hat{g}}\big)\big)\Big],\]
is algebraically equivalent to~\eqref{eq:lavin2D}.

\subsection{Linear 1D convolution via multidimensional linear convolution}
We can also compute a long 1D linear convolution with multidimensional
convolution using the technique called \emph{overlap-add}
\cite{overlap,overlap2}. For simplicity, we assume both the filter $\vcr f$ and
input $\vcr g$ are $n$-dimensional vectors. Suppose we want to decompose the
$n$-length linear convolution, where $n=\gamma \eta$, into $\gamma$ linear
convolutions for $\eta$-dimensional vectors. We represent overlap-add by the
recomposition matrix $\mat Q^{(\gamma,\eta)} \in \mathbb{R}^{2n -1 \times
(2\gamma -1)(2\eta - 1)}$, defined by
\begin{align} \label{eq:overlap_indices}
    q^{(\gamma,\eta)}_{ij} &= 
    \begin{cases} 
        1 : \text{ if }i = j - (\eta - 1)\lfloor \frac{j}{2\eta-1} \rfloor\\
        0 : \text{ otherwise}
    \end{cases}, \quad \text{with block structure} \\
    \mat Q^{(\gamma,\eta)} &=
    \begin{bmatrix}
        \mat I_{\eta-1} &&&&&&&&& \\
        & 1 &&&&&&&& \\
        && \mat I_{\eta-1} & \mat I_{\eta-1} &&&&&& \\
        &&&& 1 &&&&&\\
        &&&&& \ddots &&&&\\
        &&&&&&\mat{I}_{\eta-1}&\mat I_{\eta-1} &&\\
        &&&&&&&&1&\\
        &&&&&&&&&\mat{I}_{\eta-1}
    \end{bmatrix}.
\end{align}
\begin{theorem}
    Let $\mat{\tilde{Y}} = \mat{\tilde{F}} \ast \mat{\tilde{G}}$, where 
    $\mat{\tilde{F}},\mat{\tilde{G}} \in \mathbb{R}^{\gamma \times \eta}$.
    Then if $\vcr f = \text{vec}(\mat{\tilde{F}})$, $\vcr g =\text{vec}(\mat{\tilde{G}})$,
    $\vcr f \ast \vcr g = \text{vec}(\mat Q^{(\gamma,\eta)} \mat{\tilde Y})$.
\end{theorem}
\begin{proof}
    It suffices to show that multiplication along the last mode of
    $\tsr{T}^{(2D)}=\tsr{T}^{(\gamma)}\otimes \tsr{T}^{(\eta)}$ with $\mat
    Q^{(\gamma,\eta)}$ gives $\tsr{T}^{(\gamma\eta)}$, where we denote the
    linear convolution tensor for $n$-dimensional vectors by $\tsr{T}^{(n)}$.
    Using~\eqref{eq:overlap_indices}, we can express $\mat Q^{(\gamma,\eta)}$
    as \[q^{(\gamma,\eta)}_{a\eta+b,c(2\eta-1)+d} = \delta(a\eta+b,
    c(2\eta-1)+d-(\eta -1)c) = \delta(a\eta+b,c\eta+d),\] where $b<\eta$,
    $d<2\eta-1$, and $\delta(i,j)$ is the Kronecker delta.  Then the product of
    $\mat Q^{(\gamma,\eta)}$ and $\tsr{T}^{(2D)}$ gives
    \begin{align*}
        \sum_{c=0}^{2\gamma-2}\sum_{d=0}^{2{\eta}-2}t^{(2D)}_{i\eta+j,u\eta+v,c(2\eta-1)+d}\delta(a\eta+b,c\eta+d)
        &=\sum_{c=0}^{2\gamma-2}\sum_{d=0}^{2\eta-2} t^{(\eta)}_{iuc}t^{(\gamma)}_{jvd}\delta(a\eta+b,c\eta+d).
    \end{align*}
    We can use the definition of $\tsr{T}^{(n)}$ from~\eqref{eq:bi_func} with the
    Kronecker delta to reduce the equation above to
    \begin{align*}
        \sum_{c=0}^{2\gamma-2}\sum_{d=0}^{2\eta-2} \delta(i+u,c)\delta(j+v,d)\delta(a\eta+b,c\eta+d) 
        &= \delta(a\eta+b,(i+u)\eta+j+v) \\
        &= t^{(\gamma\eta)}_{i\eta+j,u\eta+v,a\eta+b}.
    \end{align*}
\end{proof}

\subsection{Cyclic 1D convolution via multidimensional cyclic convolution}
While the overlap-add approach decomposes a linear convolution, an $n$-length 
cyclic convolution can be broken into an $n_1 \times n_2$--length nested cyclic
convolution, where $n=n_1n_2$ and $n_1$ and $n_2$ are coprime, using the
Agarwal-Cooley Algorithm~\cite{agarwal1}. The Agarwal-Cooley algorithm uses
the Chinese remainder theorem to decompose the indices of cyclic convolution.
To denote modular arithmetic, let the notation $(x)_z$ be equivalent to $x
\text{ mod } z$. We start with the cyclic convolution between vectors
$\vcr f \in \mathbb{R}^n$ and $\vcr g \in \mathbb{R}^n$,
\begin{equation} \label{eq:circ1}
    y_k = \sum\limits_{i=0}^{n-1} f_i g_{(k-i)_n}.
\end{equation}
In order to decompose the 1D variables $k$ and $i$ into some 2D variables, we
define the corresponding modular variables, $k_1 = (k)_{n_1}, k_2 = (k)_{n_2},
i_1 = (i)_{n_1}, \text{ and } i_2 = (i)_{n_2}$.  The Chinese remainder theorem
asserts there is a unique bijection between the remainders of $k_1,k_2$ (and
similarly for $i_1,i_2$) to the original index $k$ (and similarly $i$) through
the mapping,
\[
    k = (k_1e_1 + k_2e_2)_{n} \text{ and } i = (i_1e_1 + i_2e_2)_{n},
\]
where $e_1 = n_2m_2$, $e_2 = n_1m_1$, and $m_1$ and $m_2$ are integers that
satisfy B\'{e}zout's identity~\eqref{eq:bezout_eq},
\[
    n_1m_1 + n_2m_2 = 1 \ (\text{mod } n).
\]
Therefore,~\eqref{eq:circ1} can be rewritten as
\begin{equation} \label{eq:circ2}
    \underbrace{y_{ (e_1k_1 + e_2k_2)_{n} }}_{\tilde{y}_{k_1k_2}} = 
    \sum\limits_{i_1=0}^{n_1-1} \sum\limits_{i_2 = 0}^{n_2-1} 
    \underbrace{f_{(e_1i_1 + e_2i_2)_{n}}}_{\tilde{f}_{i_1i_2}} 
    \underbrace{g_{(e_1(k_1 - i_1) + e_2(k_2 - i_2))_{n}}}_{\tilde{g}_{k_1-i_1,k_2-i_2}}.
\end{equation}
Indeed, this is now a 2D convolution problem.  We can reorder the indices
using the permutation matrix $\mat P \in \mathbb{R}^{n \times n}$, where
\[
    [\mat P]_{ij} = \begin{cases} 1 : \text{if } j =
    \lfloor i/n_2 \rfloor e_1 + (i)_{n_2}e_2 \\ 0 : \text{otherwise} \end{cases}.
\]
Now, we can apply the bilinear algorithms for the two cyclic convolution algorithms,
$(\mat{A}^{(n_1)},\mat{B}^{(n_1)},\mat{C}^{(n_1)})$
and $(\mat{A}^{(n_2)},\mat{B}^{(n_2)},\mat{C}^{(n_2)})$, and rewrite~\eqref{eq:circ2} as
\begin{equation}
    \vcr{y} = \mat{P}^\mathsf T(\mat{C}^{(n_1)} \otimes \mat{C}^{(n_2)}) 
    \big( ({\mat{A}^{(n_1)}}^\mathsf T \otimes {\mat{A}^{(n_2)}}^\mathsf T )(\mat P \vcr f) 
    \odot ({\mat{B}^{(n_1)}}^\mathsf T \otimes {\mat{B}^{(n_2)}}^\mathsf T )(\mat P \vcr g)
    \big).
\end{equation}

\subsection{Fast multidimensional convolution using low-rank approximations} \label{sec:lowrank}
Multidimensional convolution can be accelerated when the inputs to convolution
admit a low-rank matrix or tensor decomposition~\cite{khoromskij2010fast}.
We illustrate this approach for a 2D convolution of low rank matrices.
The approach extends naturally to tensors with the use of the canonical
polyadic (CP) decomposition~\cite{tensor1, khoromskij2010fast}.
For 2D convolution, suppose the input matrices $\mat F$ and $\mat{G}$ have
rank $R_{\mat F}$ and $R_{\mat G}$, respectively, so
\[
    \mat F = \sum\limits_{i=1}^{R_{\mat F}} \sigma_i^{(\vcr f)} 
    \vcr u_i^{(\vcr f)} {\vcr v_i^{(\vcr f)}}^{\mathsf T} \quad \text{and} \quad
    \mat G = \sum\limits_{i=1}^{R_{\mat G}} \sigma_i^{(\vcr g)} 
    \vcr u_i^{(\vcr g)} {\vcr v_i^{(\vcr g)}}^{\mathsf T}.
\]
Then the 2D convolution can be composed via $R_{\mat F} R_{\mat G}$ 1D convolutions,
\begin{equation}
    \mat F \ast \mat G = \sum\limits_{i=1}^{R_{\mat F}}
    \sum\limits_{j=1}^{R_{\mat G}}
    \sigma_i^{(\vcr f)} \cdot \sigma_j^{(\vcr g)} 
    \big( \vcr u_i^{(\vcr f)} \ast {\vcr u_j^{(\vcr g)}} \big)
    \big( \vcr v_i^{(\vcr f)} \ast {\vcr v_j^{(\vcr g)}} \big)^{\mathsf T}.
\end{equation}
This approach is advantageous for matrices when $R_{\mat F} R_{\mat G}<n,r$,
and is particularly valuable for convolution of tensors
with low CP rank.

\section{Fast algorithm cost comparison} \label{sec:costs}
The bilinear rank of a convolution algorithm is most important for
understanding its asymptotic complexity, especially when the algorithm is used
in a nested manner.  However, the number of additions required for computing
linear combinations is nevertheless important and typically controls the
constant-factor on the leading order term in the algorithmic cost. The
composition of the bilinear algorithm, especially for larger convolution
problems, can significantly affect the number of additions and scalar
multiplications required to apply the linear combinations. Many bilinear
algorithms exhibit an inverse relationship between the bilinear rank and the
number of flops needed for applying the linear combination~\cite{survey3}.
Different decomposition of the same convolution can lead to varied amounts of
additions in the encoding and decoding step and bilinear ranks~\cite{opt}.  In
this section, we build upon previous examinations on the number of flops
required for various compositions of bilinear algorithms~\cite{survey3,opt} by
analyzing the number of element-wise multiplications as well as flops from
linear combinations. To do so, we pay particular attention to the structure of
the matrices of the bilinear algorithms. 

\subsection{Cost bounds for general bilinear algorithms}
For bilinear algorithms without structure, as in some variants of the Toom-Cook
and Winograd's convolution algorithm, a direct computation is needed. To bound
this cost, we will study the structure of the matrices from certain Toom-Cook
and Winograd's algorithms by counting the number of non--zeros as $\text{nnz}$
in the matrices $(\mat A, \mat B, \mat C)$.

For applying a matrix--vector product $\mat A \vcr x$, we can bound the number
of additions $a(\mat A)$ and multiplications $m(\mat A)$ as
\begin{equation} \label{cref:mach1}
    a(\mat A) \le \big(\text{nnz}(\mat A) - \#\text{row}(\mat A) \big) \text{ and } 
    m(\mat A) \le \text{nnz}(\mat A).
\end{equation}
We use an upper bound since the number of non--zeros does not necessarily
correspond to additions, since some of these can be reused for later
computation. The same bound can be applied for matrices $\mat B$ and 
$\mat C$.

We represent a bilinear algorithm $F$ by its encoding and decoding matrices, $F
= (\mat A, \mat B, \mat C)$.  In general, the rank $R$ of a bilinear algorithm
$F$ is the number of columns in matrices $\mat A$ and $\mat B$. With this
notation, we can count the number of flops needed for any non-nested bilinear
algorithm as
\begin{equation} \label{cref:mach2}
    a(F) \le a(\mat A) + a(\mat B) + a(\mat C) \text{ and } 
    m(F) \le m(\mat A) + m(\mat B) + m(\mat C) + R.
\end{equation}

\subsection{Costs of fast transform algorithms}
For a bilinear algorithm where the matrices $(\mat A,\mat B,\mat C)$ have an
inherent recursive structure, a divide-and-conquer approach, such as the FFT
and DCT, can yield asymptotically fast algorithms.
For the radix-2 FFT algorithm, the cost in terms of complex additions
$\tilde{a}(n)$ and multiplies $\tilde{m}(n)$, where $T(n) =
(\tilde{a}(n),\tilde{m}(n))$, is
\begin{align*}
    T(n) &= 2T(n/2) + (n/2,n/2) \quad \text{with} \quad T(2) = (0,2), \quad \text{so} \\
    T(n) &= (n(\log(n)-1)/2,n\log(n)/2).
\end{align*}

\subsection{Costs of multidimensional methods}
Given a bilinear algorithm 
$F_{(1)} = (\mat{A}_{(1)},\mat{B}_{(1)},\mat{C}_{(1)})$ and 
$F_{(2)} = (\mat{A}_{(2)},\mat{B}_{(2)},\mat{C}_{(2)})$, let
the Kronecker product of these bilinear algorithms be
$F = F_{(1)} \otimes F_{(2)} = 
(
    \mat{A}_{(1)} \otimes \mat{A}_{(2)},\mat{B}_{(1)} \otimes \mat{B}_{(2)} ,
    \mat{C}_{(1)} \otimes \mat{C}_{(2)}
)$. 
To bound the cost of the decoding matrix $\mat A$, we can
use~\cite[Theorem~22]{opt}, 
\begin{align*} \label{eq:mach3}
    a(\mat A) &= a(\mat{A}_{(1)}) \cdot \#\text{col}(\mat{A}_{(2)}) + 
    \#\text{row}(\mat{A}_{(1)}) \cdot a(\mat{A}_{(2)}) 
    \text{ and} \\
    m(\mat A) &= m(\mat{A}_{(1)}) \cdot \#\text{col}(\mat{A}_{(2)}) + 
    \#\text{row}(\mat{A}_{(1)}) \cdot m(\mat{A}_{(2)}).
\end{align*}
This bound also applies to matrices $\mat B$ and $\mat C$. The rank of the new
bilinear algorithm $F$ is the product of the two smaller ranks, $R =
R_{(1)}R_{(2)}$. This nesting of bilinear algorithms can be extended to higher
dimensions as well. Consider a set of nested bilinear algorithms $F_{(1)},
\dots, F_{(k)}$.  We bound the cost of applying the nested linear combinations,
similar to the 2D case in~\eqref{eq:mach3}.
\begin{claim}
    Given a nested bilinear algorithm $F_{(1)}\otimes \cdots \otimes F_{(k)}$, the cost for
    encoding with matrix 
    $\mat A = \mat{A}_{(1)} \otimes \dots \otimes \mat{A}_{(k)}$, 
    where we define cost as $T(\mat A) = (a(\mat A),m(\mat A))$, is
    \begin{equation}
        T(\mat A) = \sum\limits_{i=1}^k 
        \Big( T(\mat{A}_{(i)}) \cdot \prod_{j=1}^{i-1} \#\text{row}(\mat{A}_{(j)}) 
        \cdot \prod_{j=i+1}^{k} \#\text{col}(\mat{A}_{(j)}) \Big).
    \end{equation}
\end{claim}
The same cost can be applied for encoding with the matrix 
$\mat B = \mat{B}_{(1)} \otimes \dots \otimes \mat{B}_{(k)}$ and 
for decoding with matrix
$\mat C = \mat{C}_{(1)} \otimes \dots \otimes \mat{C}_{(k)}$. 
When the rank is greater than the input size, as is the case with linear
convolution, the amount of work grows with each level of recursion.
Consequently, the cost scales exponentially to the dimension of the problem.
Given two order $d$ tensors $\tsr{F},\tsr{G} \in \mathbb{R}^{\otimes_{i=1}^d n
}$, the complexity of a direct convolution method is $O(n^{d+1})$. However, the
multidimensional FFT can compute the convolution in $O(n^d \log n)$ time. As
discussed in~\cref{sec:lowrank}, the presence of low rank structure in
$\tsr{F}$ and $\tsr{G}$ enable algorithms to circumvent the exponential
scaling in $d$.

\subsection{Fast CNN algorithm costs} \label{sec:cnn_costs}
Both the training and inference with CNNs, which rely on a series of 2D 
convolutions, are computationally intensive. As noted in the introduction, the
convolutional layer can be the most expensive step. To better understand this
cost, we will extend our cost model for bilinear algorithms to bound the
costs of the convolutional layer in CNNs.  In \cref{eq:cnn_equation}, a CNN
performs many convolutions and adds them over multiple channels~\cite{lavin},
\begin{equation} \label{eq:bilinear_cnn}
    \mat{Y}^{(i,k,\tilde x,\tilde y)} =
    \mat C \Big[ 
        \sum\limits_{c=1}^H \big(\mat{A}^\mathsf T \mat{F}^{(k,c)}\mat{A}\big)
        \odot \big(\mat{B}^\mathsf T\mat{G}^{(i,c,\tilde x, \tilde y)}\mat{B}\big) \Big] \mat{C}^\mathsf T,
\end{equation}
where the indices $\tilde x, \tilde y$ represents the different partitions of a
2D slice $\tsr{G}$ to be convolved with a 2D slice of $\tsr{F}$. Unlike
signal processing, convolutions in CNNs are associated with:
\begin{enumerate}
    \item filters (kernels) that are often much smaller than the image,
    \item filters that are reused in many convolutions,
    \item separate convolutions over multiple channels are added altogether.
\end{enumerate}
To prevent redundant transformations of both the filter and image tensor
slices, one can separately transform the filter and image and store the outputs
in the tensors $\tsr U$ and $\tsr{V}$ respectively~\cite{lavin}. The
element-wise multiplications are then computed by
\begin{equation} \label{eq:cnn_product_step}
    \mat{M}^{(i,k,\tilde x, \tilde y)} = 
    \Big[ \sum\limits_{c = 1}^H \mat{U}^{(k,c)} \odot 
    \mat{V}^{(c,l)} \Big].
\end{equation}
The tensor $\tsr{M}$ stores the output matrices from the convolutions between
every combination of the $N$ images, $K$ filters, and $P = D_HD_W/m^2$
partitions of $\tilde{x}$ and $\tilde{y}$, where $D_W \times D_H$ is the
dimension of the inputs images and $m$ is the output size of each correlation
convolution. Within the bracket of \cref{eq:cnn_product_step}, each of the $H$
channels needs to perform a convolution. Given a bilinear algorithm $F = (\mat
A,\mat B,\mat C)$ to compute the convolution of \cref{eq:cnn_product_step}, the
cost of the convolutional layer in the CNN is the sum of $T(D)$ (cost of image
transformations), $T(F)$ (cost of filter transformations), $T(I)$ (cost of
inverse transformations), and $T(M)$ (cost of the bilinear multiplications),
where 
\begin{align*} \label{eq:cnncost}
    T(F) &= KH \cdot T(\mat A), \\
    T(D) &= PNH \cdot T(\mat B), \\
    T(M) &= PKHN\cdot R^2, \text{ and} \\
    T(I) &= PKN \cdot T(\mat C).
\end{align*}
The above cost is identical to the cost model proposed
in~\cite[Equation~23]{lavin}, which is of the form,
\begin{equation}
    \alpha'(1 + \beta'/K + \gamma'/P + \delta'/H)ND_HD_WHK,
\end{equation}
where $\alpha' = R^2/m^2$, $\beta' = T(\mat B)/R^2$,
$\gamma' = T(\mat A)/R^2$, and $\delta' = T(\mat C)/R^2$.

As noted in~\cite{lavin}, the $PKHNR^2$ element-wise multiplications
in~\cref{eq:cnn_product_step} can be transformed into a multiplication 
between matrices of size $K \times H$ and $H \times PN$. By applying a fast
matrix-multiplication algorithm such as Strassen's
algorithm~\cite{Strassen_1969}, the bilinear rank of this algorithm can be
asymptotically smaller than a direct computation.

\subsection{Generating fast algorithms for CNNs}
The algorithm analyzed in~\cref{sec:cnn_costs} is one of a handful of
approaches to apply the convolutional layer. Other prominent libraries for
convolution employ an optimized direct computation or the FFT~\cite{fftw}. It
is not immediately clear which algorithm has the most optimal performance, as
experimental results~\cite{perform1,fft_better} highlight the mixed
performances of each approach. As CNNs adopt new approximation techniques such
as quantization and as scientific domains require different accuracy guarantees
for the CNN to converge, understanding trade--offs between the cost and
numerical accuracy can simplify the search for the optimal convolution
algorithm.

To quantify the costs for each convolution algorithm, it is important to
uncover both the structure of a bilinear algorithm $(\mat A,\mat B,\mat C)$ and
its bilinear rank. These values determine the overhead of the encoding/decoding 
step and the asymptotic complexity of the algorithm respectively. In the tables
below, we calculate the number of non--zeros to estimate the number of
additions and multiplies needed. We examine bilinear algorithms for varying
sizes of $n$, enforcing $n=r$ so that the two encoding matrices are identical.
Therefore, it suffices to only list the structure of $\mat A$ and $\mat C$. 

We first analyze the Toom-Cook bilinear algorithms from~\cref{sec:fastImp} and
detail their structure in~\cref{tab:results1}. This family of algorithms
encompass the popular Winograd--based algorithm for small
convolutions~\cite{lavin} and serve as a control to compare to. Next, we
consider the structure of Winograd's convolution algorithm
from~\cref{sec:fastMod}. The series of bilinear algorithm's structure and the
polynomials used to generate them are recorded in~\cref{tab:results3}. In
particular, we note that the Winograd algorithms we study here employ
at least one \textit{superlinear polynomial}~\cite{beyond}, or a polynomial
whose degree is greater than one, whereas the popular Winograd--based
methods~\cite{lavin} only use polynomials of degree one. Finally, we list the
structure of nested Toom-Cook algorithms from~\cref{sec:multidimensional}
in~\cref{tab:results2}.

\begin{table}[htb] 
\caption{(nnz, adds, mults) of Toom-Cook-based linear convolution algorithms.
    The nodes are chosen to be small integers and the $\infty$ point, e.g. for
    $n=5$, we have $0,1,-1,2$, and $\infty$.} 
\label{tab:results1} 
\centering 
\begin{tabular}{rccccc} 
\noalign{\smallskip} \hline \hline
\noalign{\smallskip} 
    $n$ & $\mat A$ & $\mat C$ & Rank \\
    \hline 
2 & $(4,1,4)$ & $(5,2,5)$ & $3$ \\
3 & $(11,6,11)$ & $(16,11,16)$ & $5$ \\
4 & $(22,15,22)$ & $(36,29,36)$ & $7$ \\
5 & $(37,28,37)$ & $(65,56,65)$ & $9$ \\
6 & $(56,45,56)$ & $(101,90,101)$ & $11$ \\ 
7 & $(79,66,79)$ & $(145,132,145)$ & $13$ \\
8 & $(106,91,106)$ & $(197,182,197)$ & $15$ \\
9 & $(137,120,137)$ & $(257,240,257)$ & $17$
\\ \noalign{\smallskip} \hline
\noalign{\smallskip} 
\end{tabular} 
\end{table} 
\begin{table}[htb] 
\caption{(nnz, adds, mults) of Winograd-based linear convolution algorithms and
    its polynomial polynomials $m^{(i)}$. For a convolution between two vectors
    of size $n$, the algorithm uses the polynomials from ``Additional
    $m^{(i)}s$'' plus the polynomials used for a convolution of size $n-1$. For
    example, the polynomials for $n=4$ include $x^2+1,x,x+1,x-1,x+2$ and
    $x-2$.}
\label{tab:results3} 
\centering 
\begin{tabular}{rcccc} 
\noalign{\smallskip} \hline \hline
\noalign{\smallskip} 
    $n$ & Additional $m^{(i)}$s & $\mat A$ & $\mat C$ & Rank \\
    \hline 
    $2$ & $x^2+1,x$ & $(5,1,5)$ & $(7,4,7)$ & $4$ \\ 
    $3$ & $x+1,x-1$ & $(13,7,13)$ & $(20,15,20)$ & $6$ \\ 
    $4$ & $x+2,x-2$ & $(25,17,25)$ & $(39,32,39)$ & $8$ \\ 
    $5$ & $x+1/2,x-1/2$ & $(41,31,41)$ & $(72,63,72)$ & $10$ \\ 
    $6$ & $x+4,x-4$ & $(61,49,61)$ & $(107,96,107)$ & $12$ \\ 
    $7$ & $x+1/4,x-1/4$ & $(85,71,85)$ & $(156,143,156)$ & $14$ \\ 
    $8$ & $x^2+2$ & $(113,96,113)$ & $(216,201,216)$ & $17$ \\ 
    $9$ & $x^2+1/2$ & $(145,125,145)$ & $(288,271,288)$ & $20$
\\ \noalign{\smallskip} \hline
\noalign{\smallskip} 
\end{tabular} 
\end{table} 
\begin{table}[htb] 
\caption{(nnz, adds, mults) of nested Toom-Cook-based linear convolution
    algorithms.  The nesting denotes the size Toom--Cook algorithms we nested
    together to construct that particular nested Toom--Cook algorithm.}
\label{tab:results2} 
\centering 
\begin{tabular}{rccccc} 
\noalign{\smallskip} \hline \hline
\noalign{\smallskip} 
    $n$ & Nesting & $\mat A$ & $\mat C$ & Rank \\
    \hline 
    $4$ & $2 \times 2$ & $(16,7,16)$ & $(25,18,25)$ & $9$ \\ 
    $6$ & $2 \times 3$ & $(44,29,44)$ & $(76,65,76)$ & $15$ \\ 
    $8$ & $2 \times 4$ & $(88,67,88)$ & $(162,147,162)$ & $21$ \\ 
    $8$ & $2 \times 2 \times 2$ & $(64,37,64)$ & $(125,110,125)$ & $27$ \\
    $9$ & $3 \times 3$ & $(121,96,121)$ & $(228,211,228)$ & $25$
\\ \noalign{\smallskip} \hline
\noalign{\smallskip} 
\end{tabular} 
\end{table} 
When comparing the three algorithms' properties, the Toom--Cook bilinear
algorithms have a lower bilinear rank than both the Winograd and nested
algorithms. While the nested algorithms have sparser encoding and decoding
matrices than both the Toom--Cook and Winograd algorithms, its rank is
significantly larger. Since the bilinear rank of $d$--dimensional problems is
$R^d$, where $R$ is the rank of the 1D bilinear algorithms used, we see the
nested algorithm may be less efficient for higher dimensional problems.
Nevertheless, for 2D and 3D problems where $n$ is not too large, the cost of
the Winograd and nested algorithms are within an additive factor of the cost
for the Toom--Cook algorithm. So, we see all three algorithms can save upwards
of $2-4 \times$ flops when compared to the DFT/FFT approach for sufficiently
small $n$.

One may question the merit of the Winograd or nested algorithm since they both
require more flops than the Toom--Cook method. The main advantage, as we
discuss more in depth in the next two sections, is that both the Winograd and
nested algorithms have improved numerical accuracy, especially when $n \geq 4$.
The underlying reason for the Toom-Cook method's instability is its use of the
Vandermonde matrix, which is ill--conditioned~\cite{vander,vander2}. Both the
Winograd and nested algorithms circumvent this issue by \textit{combining} a
series of small Toom-Cook algorithms. For the Winograd algorithm in particular,
having just one superlinear polynomial can greatly improve the
conditioning~\cite{beyond}.

In short, we find that our proposed algorithms, namely the Winograd and nested
Toom--Cook algorithms, may be beneficial when the filters are medium sized,
roughly $4 \leq r \leq 10$, and when numerical accuracy is of utmost
importance. On the other hand, for smaller problems where $r=2,3$, the
Toom--Cook method is still well--conditioned and may be the optimal algorithm
of choice. For larger problems, such as when $r \geq 20$, the FFT is the
obvious choice due to its asymptotic cost and unconditional stability.

\section{Fast algorithm accuracy comparison} \label{sec:stability}
Although fast bilinear algorithms compute convolution in asymptotically less
time than the direct approach, the use of linear combinations can introduce
considerable error from floating--point arithmetic, especially when the
multiplicative constants are large. For example, algorithms that use the
Vandermonde matrix directly for its encoding and decoding step may incur too
large of an error to be used in practice. Consequently, Toom-Cook methods 
(including the Winograd--based methods~\cite{lavin}) for medium to large
convolutions may not be stable and therefore are rarely used for inputs larger
than four \cite{winograd,lavin,nodes}.  Although a CNN can still perform well
under substantial error \cite{courbariaux2014low}, other applications of
convolution, such as in cosmology and physics, must have highly accurate
convolutions.

Well--posedness and conditioning of convolution depends on the variant of the
problem. The structured matrix formulations given
in~\cref{tab:variant_equations} are useful for reasoning about conditioning.
With the worst-case choice of both inputs, cyclic convolution is ill-posed
since if each $f_i=1$ then $\mat{C}_{\langle \vcr f \rangle}$ is rank
deficient.  On the other hand, the trapezoidal structure of $\tpl fn$ implies
that it is full rank unless $\vcr{f} = \vcr 0$.  Nevertheless, linear
convolution can be ill-conditioned for certain choices of
inputs~\cite{milinazzo1987rate}. To maintain generality across variants, we
focus on bounding the absolute error associated with a bilinear algorithm for
convolution. As with matrix products~\cite{err_bounds}, the use of mixed norms
yields a constant factor proportional to the input size. Below, we give a
simple bound for general bilinear algorithms based on only the 2-norm, similar
to the bound derived in~\cite{err1}.
\begin{theorem}[1D bilinear algorithm convolution error] \label{thm:accuracy_1d}
    Given inputs $\vcr f \in \mathbb{R}^r$ and $\vcr g \in \mathbb{R}^n$, 
    and perturbations $\delta \vcr f$, $\delta \vcr g$ such that
    $\|\delta \vcr f\| \leq \varepsilon\|\vcr f\|$ and $\|\delta \vcr g\| \leq \varepsilon\|\vcr g\|$, the
    absolute error of the bilinear algorithm $(\mat A,\mat B,\mat C)$ is
    \begin{equation}
    \begin{gathered}
        \|\delta \vcr y\| \le 
        2\big(\|\mat C\| \cdot \|\mat A\| \cdot \|\mat B\| 
        \cdot \|\vcr f\| \cdot \|\vcr g\| \big) 
        \varepsilon + O(\varepsilon^2),
    \end{gathered}
    \end{equation}
    where $\| \cdot \|$ is the 2-norm.
\end{theorem}
\begin{proof}
    We have that
    \begin{align*}
        \delta \vcr y 
        &= \mat C^T \big( 
        (\mat A \delta \vcr f) \odot (\mat B \vcr g) 
        + (\mat A \vcr f) \odot (\mat B \delta \vcr g) 
        + (\mat A \delta \vcr f) \odot (\mat B \delta \vcr g)\big).
    \end{align*}
    Now, since
    \[
        \|\vcr x \odot \vcr y\|^2 
        = \sum_i|x_iy_i|^2 
        \leq \big(\sum_i |x_i|^2\big)\big(\sum_i |y_i|^2\big) 
        = \|\vcr x\|^2\cdot \|\vcr y\|^2,
    \]
    we have  $\|\vcr x \odot \vcr y\| \leq \|\vcr x\|\cdot \|\vcr y\|.$
    Therefore,
    \begin{align*}
        \|\delta \vcr y\| 
        &\leq \|\mat C\| \cdot \big( \|\mat A \delta \vcr f\| 
        \cdot \|\mat B \vcr g\| +\| \mat A \vcr f\| \cdot \|\mat B \delta \vcr g\|\big) 
        + O(\epsilon^2),
    \end{align*}
    and the bound in the theorem follows by basic matrix and vector norm inequalities.
\end{proof}
This error bound can be extended to higher dimensions as well~\cite{err1}.
\begin{corollary} \label{cor:accuracy_nd}
    The  convolution between order $d$ inputs
    $\tsr{F} \in \mathbb{R}^{r\times \dots \times r}$ and 
    $\tsr{G} \in \mathbb{R}^{n\times \dots \times n}$
    yielding the tensor
    $\tsr{Y} = \tsr{F} \ast \tsr{G}$ using the nested 1D 
    algorithm $(\mat A,\mat B,\mat C)$ has an error of
    \begin{equation}
            ||\delta \tsr Y||_F \le 
            2\big(||\mat C||^d \cdot ||\mat A||^d \cdot 
            ||\mat B||^d \cdot ||\tsr F||_F \cdot 
            ||\tsr G||_F \big)\varepsilon + O(\varepsilon^2).
    \end{equation}
\end{corollary}
\Cref{cor:accuracy_nd} shows that the error is proportional to the norm of the
bilinear algorithm's matrices $(\mat A,\mat B,\mat C)$ and exponential to the
dimension of the problem. For algorithms like the Toom-Cook
method~\cref{sec:toomcook}, $\mat{A}$ and $\mat{B}$ are submatrices of a
Vandermonde matrix and $\mat{C}$ is its inverse. Consequently, the absolute
error of Toom-Cook convolution scales with the condition number of the
Vandermonde matrix. When the nodes used for interpolation--based methods are
restricted to real values, the condition number of the resulting Vandermonde
matrix will be exponential in its dimension~\cite{vander}. Therefore, the
Toom-Cook method with real--valued interpolation nodes will produce encoding
matrices whose norm is exponential to the problem size. Selecting complex nodes
can fix the ill-conditioning (e.g., via DFT~\cref{sec:transforms}), but smarter
selections of real nodes can also somewhat improve the conditioning.

\subsection{Improved accuracy by nodes and scaling} 
Chebyshev nodes are real--valued nodes that can improve the conditioning of the
Vandermonde matrix without requiring additional costs. While these nodes
produce matrices with smaller condition numbers than that of integer nodes, the
condition number still grows exponentially with respect to the dimension of the
inputs~\cite{vander2}. Empirical experiments based on exhaustive search show
that choosing nodes with few significant mantissa bits and ``symmetric'' nodes,
or nodes that are the negative, reciprocal, and negative reciprocal of
previously chosen nodes, will yield better conditioned matrices~\cite{err1}.
For instance, the norm of the Vandermonde matrix can remain relatively low when
selecting the points $2, -2, 1/2, \text{ and } -1/2$. Finding a ``widely
accepted strategy for selecting [good] points'' without using the complex
domain is an open question~\cite{err1}.

Another technique to improve accuracy is diagonal scaling~\cite{nodes,diag1}.
Diagonal scaling introduces a diagonal matrix multiplication to each of the
matrices $(\mat A,\mat B,\mat C)$ while preserving the correct convolution
output. This scaling reduces the magnitude of the entries in the matrices,
which can improve the condition number of the matrices.  By empirically
identifying the best weights for the diagonal matrix, diagonal scaling reduces
the maximum relative error of convolution. Experimental results of AlexNet
show that for a correlation convolution algorithm with filter of size $r=5$ and
output of size $m=9$ and using well-chosen nodes, diagonal scaling can reduce
the maximum relative error from $7.53 \times 10^{-2}$ to $5.49 \times
10^{-4}$~\cite{nodes}. By comparison, a direct computation of the correlation
algorithm achieves a maximum relative error of $2.81 \times 10^{-6}$.

\subsection{Improved accuracy by small nested convolutions} 
When highly accurate convolution algorithms are needed, well--chosen nodes may
not offer enough norm reductions to significantly reduce the error. Instead,
another strategy, proposed in signal processing, is to break a long convolution
into a series of smaller ones. To illustrate why this works, recall that the
condition number of an $n\times n$ Vandermonde matrix is exponential to its
input size. When selecting integer points, the conditioning of the Vandermonde
matrix $V$ is $\Omega\big(n^n \big)$~\cite{vander}. Instead, if the bilinear
algorithm is decomposed from an $n = n_1n_2$-length convolution into a sequence
of $n_1$-sized convolution nested with $n_2$-sized convolution, the condition
number of the nested Vandermonde matrix by a Kronecker product is
$\Omega\big((n_1+ n_2)^{n_1+n_2}\big)$. By repeating this decomposition, the
accuracy of fast convolution algorithms that rely on the Vandermonde matrix can
be greatly improved.

In order to devise such nested algorithms, we can employ the overlap-add
approach for linear convolution and the Agarwal-Cooley algorithm for cyclic
convolution from \cref{sec:multidimensional}. The error bound of using the
Agarwal-Cooley algorithm is identical to \cref{cor:accuracy_nd}, as the
additional permutation matrices $\mat{P}$ have a norm of $1$.  For the
overlap-add approach, the recomposition matrix $\mat{Q}^{(\gamma,\eta)}$
introduces some floating-point error due to its additions at the end of each
nested convolution. This error is relatively small as long as the dimension
size is not too large, as shown by the following bound.
\begin{theorem} \label{thm:overlap_err_bound}
    Given inputs $\vcr{f} \in \mathbb{R}^{n^d}$ and 
    $\vcr{g} \in \mathbb{R}^{n^d}$ 
    and perturbations $\delta \vcr f$, $\delta \vcr g$ such that
    $\|\delta \vcr f\| \leq \varepsilon\|\vcr f\|$ and $\|\delta \vcr g\| \leq \varepsilon\|\vcr g\|$,
    the convolution of $n$-dimensional vectors based on the 1D linear convolution bilinear algorithm
    $(\mat A, \mat B, \mat C)$ nested using the overlap-add method
    has an error of
    \[
        ||\delta \vcr y|| \le 2^{d/2 + 1} \cdot ||\mat C||^d \cdot 
        ||\mat A||^d \cdot ||\mat B||^d \cdot
        ||\vcr f|| \cdot ||\vcr g|| \cdot \varepsilon 
        + O(\varepsilon^2).
    \]
\end{theorem}
\begin{proof}
    Let $\mat{Q}^{(i)} = \mat{Q}^{(\gamma^{(i)},\eta^{(i)})}$ be the overlap-add
    matrix for the $i$th level of the nested bilinear algorithm. Using
    \cref{cor:accuracy_nd}, we have that to first order in $\varepsilon$,
        \begin{align*}
        ||\delta \mat y|| &\le 
        2 ||\mat{Q}^{(d)} \mat C \otimes \dots \otimes \mat{Q}^{(1)} \mat C||
        \cdot || \mat A \otimes \dots \otimes \mat A || 
        \cdot 
        || \mat B \otimes \dots \otimes \mat B || \cdot || \vcr f || 
        \cdot ||\vcr g|| \cdot \varepsilon.
        \end{align*}
    Notice that $ || \mat{Q}^{(i)} || \le \sqrt{2}$ since
    each row has at most two ones. Simplifying leads to the bound in the
    theorem.
\end{proof}

\subsection{Orthogonal polynomials as a basis} 
Decomposing a long convolution into a series of small nested convolutions can
help us achieve highly accurate convolution. For cases where we need very
accurate convolutions, one approach is to simply use the DFT. The discrete
Fourier matrix has bounded conditioning, making it the ideal choice when
accuracy is imperative.  For cases where we want the same accuracy without use
of complex arithmetic, we can instead use orthogonal polynomials.

By using orthogonal polynomials to define the encoding and decoding matrices
$(\mat A,\mat B,\mat C)$, the resulting matrices are generally
well-conditioned. The trade--off is that the input must be converted to and
from its monomial basis to the orthogonal basis. For certain orthogonal
polynomials, this conversion introduces large multiplicative
scalars~\cite{ortho}, thereby negating the accuracy of the orthogonal
polynomials.  One approach to circumvent the cost of basis transformation is to
leverage the embedding of convolution in a monomial basis within convolution in
a Chebyshev basis, which corresponds to the use of DCT for linear convolution
(described in~\cref{subsec:dct}).

\section{Numerical experiments} \label{sec:experiments}
We provide experimental results on the numerical accuracy of the following
bilinear algorithms for linear convolution: Toom-Cook with integer nodes,
Toom-Cook with Chebyshev nodes, Winograd convolution algorithm with superlinear
polynomial divisors, and the nested Toom-Cook method. All the code is written
in Python with NumPy. The inputs are composed of randomly chosen real numbers
from the set $[0,1)$.  We use NumPy's \texttt{seed()} function with a seed of
$1$ to ensure these results are reproducible. To calculate the relative error, we
compute a convolution from a bilinear algorithm using
\texttt{compute\_bilinear\_algorithm()} and compare it with the convolution
from a direct computation using \texttt{direct\_conv()}\footnote{Methods
available in \textit{test.py} from \\
\url{https://github.com/jucaleb4/Bilinear-Algorithms-for-Convolution}}. 

\begin{figure}
\centering
\begin{subfigure}{.5\textwidth}
  \centering
  \includegraphics[width=1.0\linewidth]{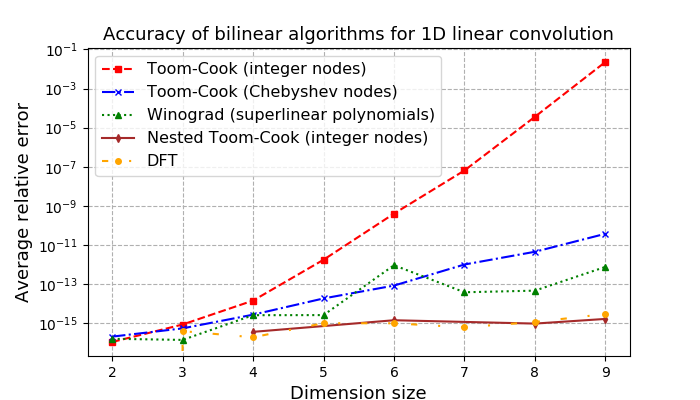}
  \caption{1D convolution}
  \label{fig:fig1a}
\end{subfigure}%
\begin{subfigure}{.5\textwidth}
  \centering
  \includegraphics[width=1.0\linewidth]{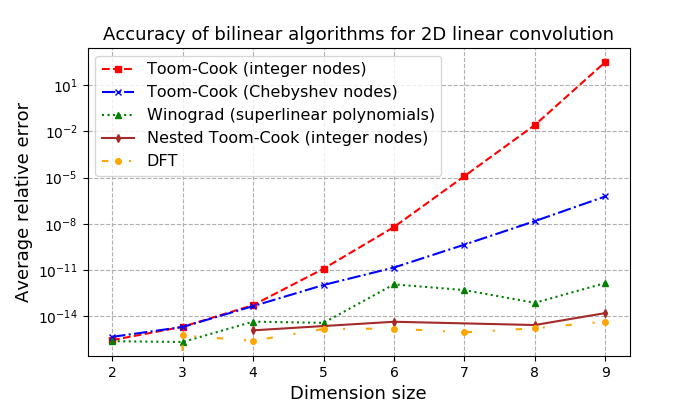}
  \caption{2D convolution}
  \label{fig:fig1b}
\end{subfigure}
\centering
\begin{subfigure}{.5\textwidth}
  \centering
  \includegraphics[width=1.0\linewidth]{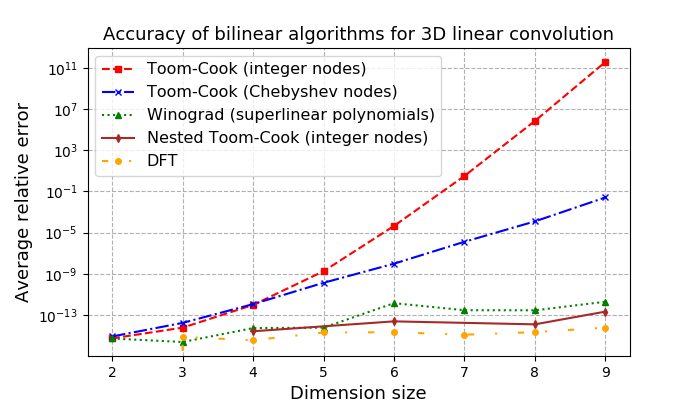}
  \caption{3D convolution}
  \label{fig:fig1c}
\end{subfigure}%
\begin{subfigure}{.5\textwidth}
  \centering
  \includegraphics[width=1.0\linewidth]{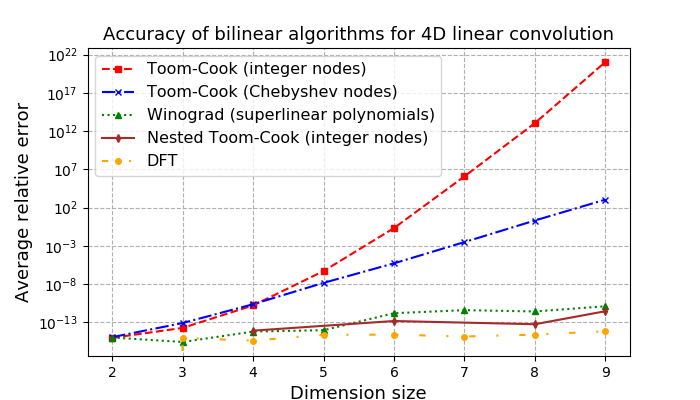}
  \caption{4D convolution}
  \label{fig:fig1d}
\end{subfigure}
    \caption{The relative error (averaged over ten trials) 
    of the linear convolution between (random) $d$--dimensional 
    tensors ($d=1,\ldots,4$) over various mode lengths ($n=2,\ldots,9$). For
    any dimension and mode length, the same pair of tensors were used in all
    convolution algorithms.}
\end{figure}

\subsection{Accuracy of the Toom-Cook method}
\Cref{fig:fig1a,fig:fig1b,fig:fig1c,fig:fig1d} show the relative error of the
Toom-Cook method using small integer and Chebyshev nodes. Both methods incur
substantial errors when the input size exceeds size six, especially as the
dimension of the problem increases. For multidimensional problems, the use of
Chebyshev nodes can significantly reduce the relative error from using integer
nodes. However, we observe that once the algorithm is used for 3D or 4D
convolution with inputs greater than size seven, the use of Chebyshev nodes
still leads to high errors.

\subsection{Accuracy of the Winograd convolution algorithm}
We implement Winograd's convolution algorithm based on the formulation
from~\cref{sec:fastMod} using the list of polynomial divisors $m^{(i)}$
from~\cref{tab:results3} and plot the average relative errors
in~\cref{fig:fig1a,fig:fig1b,fig:fig1c,fig:fig1d}. We observe that Winograd's
convolution algorithm with just one superlinear polynomial divisor (increasing
rank by 1 with respect to optimum), as is the case for convolutions of size up
to $6$, can significantly reduce the relative error compared to the Toom-Cook
method with Chebyshev nodes.  Furthermore, the number of flops required is only
marginally larger than that of the Toom-Cook method.  For a $5$-dimensional
convolution, the number of additions, multiplications, and rank of Winograd's
convolution algorithm (\cref{tab:results3}) and the Toom-Cook method
(\cref{tab:results1}) suggests that Winograd's convolution algorithm can
achieve highly accurate results without a significant increase in arithmetic or
use of complex arithmetic as with the DFT.

\subsection{Accuracy of the nested Toom-Cook method}
We show the accuracy of the nested Toom-Cook method
in~\cref{fig:fig1a,fig:fig1b,fig:fig1c,fig:fig1d}.  Like the Winograd
convolution algorithm, this algorithm can significantly reduce the error of the
Toom-Cook method. A downside to the nested Toom-Cook approach is that the
bilinear rank is greatly increased as compared to Winograd's convolution
algorithm. For example, our implementation of a Winograd's convolution
algorithm for $8$-dimensional vectors has a rank of $17$, whereas the nested
Toom-Cook algorithm has a bilinear rank of $27$. 

However, a benefit of the nested Toom-Cook method is that the number of
non--zeros in both the encoding and decoding matrices are lower than Winograd's
convolution algorithm, as shown seen in~\cref{tab:results2}
and~\cref{tab:results3}. For certain decompositions, the magnitude of the
matrix elements never exceeds $1$, as these matrices are built from very small
Toom-Cook methods. A nested Toom-Cook algorithm for $n=8$ can be created by a
triply nested Toom-Cook algorithm for $n=2$, whose matrices are composed of
zeros and the scalars $1$ and $1/2$ as well as their negatives. For
convolutions where the overhead of applying the encoding and decoding is
computationally expensive (such as near the leaves of the recursion tree), the
nested Toom-Cook approach offers an accurate and efficient approach.

\section{Future work} \label{sec:future}
While novel strategies to reduce error from~\cref{sec:stability}, such as
better node points and diagonal scaling, have the potential to improve the
accuracy of the convolution algorithm by a constant factor~\cite{err1,nodes},
it remains an open question on how much these strategies can improve accuracy.
A better understanding of the round-off error as well as finding an
algorithm that determines the set of nodes and diagonal scaling with the
optimal norm can help extend the use of Toom-Cook convolution algorithms to
larger filter sizes. Similarly, there remains space for a more comprehensive
search of divisor polynomials to produce encoding and decoding matrices with an
optimal balance between sparsity, rank, and conditioning for the Winograd
convolution algorithm. Experimental results suggest that polynomial divisors
that are superlinear and added with positive/negative powers of two offer
robust numerical accuracy~\cite{beyond}. Theory that supports this claim or
finds other suitable polynomials can be of interest.

Furthermore, low--precision multipliers in deep neural networks are robust for
both training and inferencing~\cite{courbariaux2014training}. However,
existing low--precision training strategies rely on \textit{high--precision}
convolutions or dot products. Especially with the emergence of mixed--precision
training on accelerators such as Tensor Cores, deriving theoretical bounds as
well as empirical results on the interplay between low--precision arithmetic
and the accuracy of the various bilinear algorithms can be of interest to the
deep learning and high--performance computing community.

Another area for future study is the design and analysis of fast parallel 
convolution algorithms. Performance of convolution algorithms in the parallel
setting (such as on GPUs) is generally dominated by communication costs, i.e.,
the amount of data movement required to compute the algorithm. As the size of
the dataset grows, more communication is needed. Demmel and Dinh derived 
communication lower bounds for the direct computation of the convolutional
layer~\cite{comm_opt}. To extend their work, an open question is determining
how much communication is needed for fast bilinear algorithms such as the
Toom-Cook algorithm, Winograd's convolution algorithm, and Toom-Cook with
overlap-add. General approaches for deriving communication lower bounds of
bilinear algorithms~\cite{SDH_ETHZ_2015} may provide one avenue towards
understanding communication costs in fast convolution algorithms.  Moreover, a
few of the convolution algorithms we consider (the one derived
in~\cref{sec:sparse_bilinear} and the many--nested Toom-$2$ algorithm) can be
formulated by taking products with sparse matrices.  Exploration of efficient
sparse linear algebra kernels for these convolution variants may yield
convolution kernels with better performance and accuracy.

Another question is whether the interpolation and Winograd  techniques covered
in this paper encompass all possible fast bilinear algorithms for convolution.
If these techniques do cover all possible fast bilinear algorithms, this
knowledge can narrow the search for optimal bilinear algorithms in accuracy and
number of flops.

\section{Conclusion} \label{sec:conclusions}
Using the formalism of bilinear algorithms, we present different variants of
convolution, including ones based on polynomial interpolation and modular
polynomial arithmetic. We derive simple formulations for generating these
bilinear algorithms. These explicit formulations allow us to quantify the cost
of the various algorithms as well as simplify previous error bounds. Our
analysis and experiments show that the nested convolution via overlap-add and
Winograd's convolution algorithm with superlinear polynomials can be effective
for multidimensional convolution for a range of filter sizes. With the
simplified construction of these convolution algorithms in the language of
linear algebra, we hope researchers in scientific computing, applied
mathematics, and machine learning can discover new uses and methods for fast
convolution.

\section*{Acknowledgments}
We would like to thank Hung Woei Neoh for helpful discussions and the
anonymous referees for providing valuable feedback that helped improve
this manuscript. 

\bibliographystyle{siamplain} \bibliography{references} \end{document}

%% file: ex_shared.tex
\usepackage{lipsum}
\usepackage{amsfonts}
\usepackage{graphicx}
\usepackage{epstopdf}
\usepackage{bm}
\usepackage{algorithmic}
\usepackage{todonotes}
\usepackage{nccmath}
\usepackage{relsize}
\usepackage{hyperref}
\usepackage{url}
\ifpdf
  \DeclareGraphicsExtensions{.eps,.pdf,.png,.jpg}
\else
  \DeclareGraphicsExtensions{.eps}
\fi
\usepackage[normalem]{ulem}

\newsiamremark{remark}{Remark}
\newsiamremark{hypothesis}{Hypothesis}
\crefname{hypothesis}{Hypothesis}{Hypotheses}
\newsiamthm{claim}{Claim}

\usepackage{amsopn}

\makeatletter
\newcommand*{\addFileDependency}[1]{
  \typeout{(#1)}
  \@addtofilelist{#1}
  \IfFileExists{#1}{}{\typeout{No file #1.}}
}
\makeatother

\newcommand*{\myexternaldocument}[1]{%
    \externaldocument{#1}%
    \addFileDependency{#1.tex}%
    \addFileDependency{#1.aux}%
}

\newcommand{\adots}{\reflectbox{$\ddots$}}

\newcommand{\vcr}[1]{\bm{#1}}
\newcommand{\mat}[1]{\bm{#1}}
\newcommand{\tsr}[1]{\bm{\mathcal{#1}}}
\newcommand{\fun}[1]{#1}
\newcommand{\tpl}[2]{\mat{T}_{\langle\vcr{#1},#2\rangle}}
\newcommand{\hkl}[2]{\mat{H}_{\langle\vcr{#1},#2\rangle}}

\newcommand{\xmm}[2]{\mat{X}_{\langle #1,#2\rangle}}
\renewcommand{\deg}[1]{\text{deg}(#1)}